\documentclass[a4paper, pdftex]{amsart}

\AtBeginDocument{%
  \def\MR#1{}
}

\usepackage{tikz, tikz-cd}
\usepackage{amsmath,amssymb}
\usepackage{amsthm}
\usepackage{mathrsfs}
\usepackage[shortlabels]{enumitem}
\usepackage{mathtools}
\usepackage[all,cmtip,2cell]{xy}
\usepackage{array}
\usepackage{comment}
\usepackage{braket}
\usepackage{stmaryrd}
\usepackage{url}
\usepackage{ascmac}

\usepackage{amscd,bm,graphicx,tikz-cd,upgreek,dsfont,amsfonts,tikz}

\usepackage[top=30truemm,bottom=30truemm,left=35truemm,right=35truemm]{geometry}

\theoremstyle{plain}
\newtheorem{thm}{Theorem}[section]
\newtheorem{prop}[thm]{Proposition}
\newtheorem{lem}[thm]{Lemma}
\newtheorem{claim}[thm]{Claim}
\newtheorem{cor}[thm]{Corollary}
\newtheorem*{thm*}{Theorem}

\theoremstyle{definition}
\newtheorem{defi}[thm]{Definition}

\newtheorem*{NaC}{Notation and Convention}
\newtheorem*{ACK}{Acknowledgment}

\theoremstyle{remark}
\newtheorem{rem}[thm]{Remark}
\newtheorem{nota}[thm]{Notation}

\numberwithin{equation}{section}

\newcommand{\Z}{\mathbb{Z}}

\renewcommand{\P}{\mathbb{P}}

\renewcommand{\a}{\alpha}

\newcommand{\e}{\varepsilon}

\newcommand{\s}{\sigma}
\renewcommand{\t}{\tau}
\newcommand{\vp}{\varphi}

\newcommand{\emp}{\varnothing}

\newcommand{\ol}{\overline}
\newcommand{\wt}{\widetilde}

\newcommand{\ra}{\Rightarrow}
\newcommand{\la}{\Leftarrow}

\newcommand{\hra}{\hookrightarrow}
\newcommand{\epm}{\twoheadrightarrow}

\DeclareMathOperator{\rk}{rk}
\DeclareMathOperator{\id}{id}

\DeclareMathOperator{\Sym}{\mathrm{Sym}}

\DeclareMathOperator{\Spec}{\mathrm{Spec}}

\DeclareMathOperator{\Proj}{\mathrm{Proj}}

\DeclareMathOperator{\Pic}{\mathrm{Pic}}
\DeclareMathOperator{\Gr}{\mathrm{Gr}}
\DeclareMathOperator{\Fl}{\mathrm{Fl}}

\DeclareMathOperator{\Bl}{\mathrm{Bl}}
\DeclareMathOperator{\Br}{\mathrm{Br}}
\DeclareMathOperator{\codim}{\mathrm{codim}}

\DeclareMathOperator{\pr}{pr}
\DeclareMathOperator{\ch}{ch}

\DeclareMathOperator{\Hom}{Hom}
\DeclareMathOperator{\End}{End}
\DeclareMathOperator{\Ker}{Ker}
\DeclareMathOperator{\Cok}{Cok}

\DeclareMathOperator{\Ext}{Ext}

\DeclareMathOperator{\Hilb}{Hilb}
\DeclareMathOperator{\Univ}{\mathrm{Univ}}
\DeclareMathOperator{\ev}{ev}
\DeclareMathOperator{\coev}{coev}

\newcommand{\mcA}{\mathcal{A}}
\newcommand{\mcB}{\mathcal{B}}

\newcommand{\mcD}{\mathcal{D}}
\newcommand{\mcE}{\mathcal{E}}
\newcommand{\mcF}{\mathcal{F}}
\newcommand{\mcG}{\mathcal{G}}
\newcommand{\mcH}{\mathcal{H}}
\newcommand{\mcI}{\mathcal{I}}

\newcommand{\mcK}{\mathcal{K}}
\newcommand{\mcO}{\mathcal{O}}
\newcommand{\mcL}{\mathcal{L}}
\newcommand{\mcN}{\mathcal{N}}
\newcommand{\mcM}{\mathcal{M}}
\newcommand{\mcQ}{\mathcal{Q}}
\newcommand{\mcR}{\mathcal{R}}

\newcommand{\mcT}{\mathcal{T}}

\newcommand{\mcV}{\mathcal{V}}

\let\Im\relax
\DeclareMathOperator{\Im}{\mathrm{Im}}

\DeclareMathOperator{\Coh}{Coh}

\DeclareMathOperator{\RG}{\mathrm{R}\Gamma}
\DeclareMathOperator{\ext}{\mathrm{ext}}

\DeclareMathOperator{\RHom}{\mathrm{RHom}}
\newcommand{\RR}{\mathbf{R}}
\newcommand{\LL}{\mathbf{L}}

\DeclareMathOperator{\Db}{\mathrm{D}^{\mathrm{b}}}

\newcommand{\wF}{\mathsf{wF}}
\newcommand{\ins}{\mathsf{ins}}

\newcommand{\Sch}{\mathrm{Sch}}
\newcommand{\AffSch}{\mathrm{AffSch}}

\newcommand{\Sets}{\mathrm{Sets}}
\newcommand{\op}{\mathrm{op}}

\renewcommand{\mod}{\mathrm{mod}}

\mathchardef\mhyphen="2D

\title[Rank $2$ Weak Fano bundles on del Pezzo $3$-folds of degree $5$]{Rank two weak Fano bundles on del Pezzo threefolds of degree five}
\author[T.FUKUOKA, W.HARA, D.ISHIKAWA]{Takeru Fukuoka, Wahei Hara, Daizo Ishikawa}

\address[T.FUKUOKA]{Graduate School of Mathematical Sciences\\The University of Tokyo\\3-8-1 Komaba\\Meguro-ku, Tokyo 153-8914, Japan}
\email{tfukuoka@ms.u-tokyo.ac.jp}

\address[W.HARA]{The Mathematics and Statistics Building, University of Glasgow, University Place, Glasgow, G12 8QQ, UK.}
\email{wahei.hara@glasgow.ac.uk}

\address[D.ISHIKAWA]{Department of Mathematics, School of Science and Engineering, Waseda University, Ohkubo 3-4-1, Shinjuku, Tokyo 169-8555, JAPAN}
\email{azoth@toki.waseda.jp}

\date{\today}


\begin{document}
\maketitle
\begin{abstract}
This paper classifies rank two vector bundles on a del Pezzo threefold $X$ of degree five whose projectivizations are weak Fano. 
This classification is then used to determine properties of the moduli spaces of such vector bundles on $X$, and we determine precisely when the moduli spaces are smooth, irreducible, and fine.
We also prove that such a bundle on a del Pezzo threefold of degree one or two splits, and as result give a classification of weak Fano bundles of rank two on a del Pezzo threefold of Picard rank one. 
\end{abstract}

\section{Introduction}


\subsection{Background}
To deepen the classification theory of Fano manifolds, Szurek and Wisniewski \cite{SW90} introduced the notion of a \emph{Fano bundle}, which is a vector bundle $\mcE$ on a smooth projective variety $X$ such that $\P_{X}(\mcE)$ is Fano. 
They showed that the base space $X$ of a Fano bundle $\mcE$ must be a Fano manifold \cite[THEOREM]{SW90}. 
In many papers \cite{SW90,Langer98,mos2,mos1,Ohno16v7}
the classification problem of Fano bundles have been treated on a fixed base Fano manifold,
and rank $2$ Fano bundles on a Fano $3$-fold of Picard rank $1$ were classified by \cite{mos2}. 

In \cite{Langer98}, Langer introduced the notion of \emph{weak Fano} bundles as a natural generalization of that of Fano bundles. 
Namely, a vector bundle $\mcE$ on $X$ is said to be weak Fano if $\P_{X}(\mcE)$ is weak Fano, i.e., $-K_{\P_{X}(\mcE)}$ is nef and big. 
As opposed to Fano bundles, rank $2$ weak Fano bundles on a Fano $3$-fold of Picard rank $1$ have not yet been classified. 

Our previous articles \cite{Ishikawa16,FHI20} classified rank $2$ weak Fano bundles on a del Pezzo $3$-fold $X$ when the degree of $X$, say $\deg X$, is $3$ or $4$.
Recall that a \textit{del Pezzo $3$-fold} is a smooth Fano $3$-fold whose Fano index $r(X) :=\max \{r \mid -K_{X} \sim rH \text{ for some } H \in \Pic(X)\} $ is $2$, and its degree is defined by $\deg X = H^{3}$, where $H$ is the divisor that attains $-K_{X} \sim 2H$.
It is known that a del Pezzo threefold $X$ satisfies $\deg X \leq 5$ if and only if the Picard rank of $X$ is equal to $1$ \cite{Iskovskikh77,Fujita80}.

In the sequel of the previous works, the present paper classifies rank two weak Fano bundles over the del Pezzo $3$-fold $X$ of $\deg X=5$. 
The classification for $\deg X \in \{1,2\}$ cases will follow easily. 
Therefore, as a consequence, this paper completes the classification of rank $2$ weak Fano bundles on a del Pezzo $3$-fold $X$ of Picard rank $1$. 


\subsection{Known Results}\label{subsec-intro-resol-known}

Our previous work \cite{FHI20} treated the case when $\deg X=4$,
where roughly speaking, we showed that every indecomposable rank $2$ weak Fano bundle is either a Fano bundle or a vector bundle $\mcE$ with $c_{1}(\mcE)=0$ such that $\mcE(1)$ is globally generated. 
By the same approach as in \cite{FHI20}, we will see that this phenomenon still holds when $\deg X = 5$ (see Section~\ref{sec-conclusions}),
although the techniques here to establish the results are very different (see Section \ref{subsec-PlanOfProof}).

In order to deepen the study of rank $2$ weak Fano bundles $\mcE$ when $\deg X=5$, 
we give a new description of such bundles in terms of resolutions consisting of natural vector bundles on $X$. 
This approach has been used often in the study of vector bundles on projective spaces: indeed, \cite{Langer98}, \cite{SW90} and \cite{Ohno16v7} gave the classification of vector bundles in terms of linear resolutions of them. 
This form of classification is useful for studying various properties of vector bundles such as its global generation or properties of its moduli space. 

As in the case $\deg X = 4$, an instanton bundle $\mcE$ such that $\mcE(1)$ is globally generated and $c_{2}(\mcE) \leq 4$ is an example of a weak Fano bundle.
Recall that an \emph{instanton bundle} on a del Pezzo $3$-fold is a slope stable rank $2$ vector bundle $\mcE$ such that $c_{1}(\mcE)=0$ and $h^{1}(\mcE(-1))=0$ \cite{Kuznetsov12,Faenzi14}.
While instanton bundles $\mcE$ on a del Pezzo $3$-fold of degree $5$ were studied by Sanna and investigated very effectively when $c_{2}(\mcE) \leq 3$ \cite{Sanna14,Sanna17}, 
those with $c_{2}(\mcE)=4$ is much less well understood. 
This article contains both new results for $c_{2}(\mcE) \leq 3$, and also a description of instanton bundles $\mcE$ with $c_{2}(\mcE) = 4$ such that $\mcE(1)$ is globally generated.
Applying those results reveals some properties for their moduli spaces.



\subsection{Main Results}\label{subsec-intro-resol-dP5}

The main results of this article are broadly divided into three parts.

\subsubsection{Classification on a del Pezzo threefold of degree $5$.}

Let $X$ be a del Pezzo $3$-fold of degree $5$. 
The first result of this article is the classification of rank $2$ weak Fano bundles on $X$. 

By classical classification theory of del Pezzo $3$-folds due to Fujita \cite{Fujita81} and Iskovskikh \cite{Iskovskikh77}, $X$ is a codimension $3$ linear section of $\Gr(2,5)$. 
Let $\mcR$ and $\mcQ$ denote the restriction of the universal rank $2$ subbundle and the universal quotient bundle of rank $3$ on $\Gr(2,5)$, respectively. 
The main result of this article is the following explicit resolution of rank $2$ weak Fano bundles on $X$. 

\begin{thm}\label{mainthm-resol}
Let $X$ be a del Pezzo $3$-fold of degree $5$. 
Then, every weak Fano bundle $\mcE$ is isomorphic to one of the following up to twisting with a line bundle. 
\renewcommand{\labelenumi}{(\roman{enumi})}
\begin{enumerate}
\item $\mcE \simeq \mcO_{X}(1) \oplus \mcO_{X}(-1)$. 
In this case, $c_{1}(\mcE)=0$ and $c_{2}(\mcE)=-5$. 
\item $\mcE \simeq \mcO_{X} \oplus \mcO_{X}(-1)$. 
In this case, $c_{1}(\mcE)=-1$ and $c_{2}(\mcE)=0$. 
\item $\mcE \simeq \mcO_{X}^{\oplus 2}$.
In this case, $c_{1}(\mcE)=0$ and $c_{2}(\mcE)=0$. 
\item $\mcE \simeq \mcR$.
In this case, $c_{1}(\mcE)=-1$ and $c_{2}(\mcE)=2$. 
\item $\mcE$ fits into 
$0 \to \mcQ(-1) \to \mcO_{X} \oplus \mcR^{\oplus 2} \to \mcE \to 0$. 
In this case, $c_{1}(\mcE)=0$ and $c_{2}(\mcE)=1$.
\item $\mcE$ fits into 
$0 \to \mcQ(-1)^{\oplus 2} \to \mcR^{\oplus 4} \to \mcE \to 0$. 
In this case, $c_{1}(\mcE)=0$ and $c_{2}(\mcE)=2$.
\item $\mcE$ fits into 
$0 \to \mcO_{X}(-1) \oplus \mcQ(-1) \to \mcR^{\oplus 5} \to \mcO_{X}^{\oplus 8} \to \mcE(1) \to 0$. 
In this case, $c_{1}(\mcE)=0$ and $c_{2}(\mcE)=3$.
\item $\mcE$ fits into 
$0 \to \mcO_{X}(-1)^{\oplus 2} \to \mcQ(-1)^{\oplus 2} \to \mcO_{X}^{\oplus 6} \to \mcE(1) \to 0$. 
In this case, $c_{1}(\mcE)=0$ and $c_{2}(\mcE)=4$.
\end{enumerate}

In the above results, we regard the numerical class groups of $\Z$-coefficients $N^{1}(X)_{\Z}$ and $N^{2}(X)_{\Z}$ as $\Z$ by taking the effective classes generating these class groups.

Furthermore, on an arbitrary del Pezzo $3$-fold of degree $5$, 
there exist examples for each case of (i)--(viii). 
\end{thm}

\begin{rem}\label{mainrem-resol}
The resolutions given in Theorem~\ref{mainthm-resol} are not unique. 
For example, there exist the following alternative resolutions. 

\renewcommand{\labelenumi}{(\roman{enumi})}
\begin{enumerate}
\setcounter{enumi}{4}
\item The exact sequence in (v) can be replaced by 
$0 \to \mcR \to \mcQ^{\vee} \oplus \mcO_{X} \to \mcE \to 0$. 
For more details, see Remark~\ref{rem-c21-another}. 
\item The exact sequence in (vi) can be replaced by
$0 \to \mcR^{\oplus 2} \to (\mcQ^{\vee})^{\oplus 2} \to \mcE \to 0$. 
For more details, see Section~\ref{subsec-c2=2} and the sequence (\ref{ex-c22-another}). 
\item The exact sequence in (vii) can be replaced by
$0 \to \mcO_{X}(-1) \to \mcR^{\oplus 2} \oplus \mcQ^{\vee} \to \mcO_{X}^{\oplus 8} \to \mcE(1) \to 0$.
For more details, see Remark~\ref{rem-c23-another}. 
\end{enumerate}
\end{rem}

\begin{rem}\label{rem-stability}
By our classification and Hoppe's criterion \cite[Theorem~2.10]{Sanna17}, $\mcE$ is slope stable (resp. slope semi-stable, not slope semi-stable) if and only if $\mcE$ is of type (iv), (vi), (vii) or (viii) (resp. (iii) or (v), (i) or (ii)) (see also Remark~\ref{rem-st}). 
%
%
\end{rem}

\begin{rem}
The known classification of Fano bundles \cite{mos2} implies that, among the list in Theorem~\ref{mainthm-resol}, the bundles of type (ii), (ii) and (iv) are Fano, and others are not Fano but weak Fano.
\end{rem}

\subsubsection{Moduli spaces}

Our description of rank $2$ weak Fano bundles on a del Pezzo $3$-fold $X$ of degree $5$ (=Theorem~\ref{mainthm-resol}) is useful to investigate their moduli spaces. 
Let $M^{\wF}_{c_{1},c_{2}}$ be the coarse moduli space of rank $2$ weak Fano bundles $\mcE$ on $X$ with $c_{1}(\mcE)=c_{1}$ and $c_{2}(\mcE)=c_{2}$ (see Definition~\ref{defi-modulispace}). 
Since the classes in Theorem~\ref{mainthm-resol}~(i),(ii), (iii), and (iv) consist of the unique object, the moduli problem for them is not interesting.
For the class in Theorem~\ref{mainthm-resol}~(v), Proposition~\ref{prop-moduli-prelim}~(1) will show that
the coarse moduli space 
$M^{\wF}_{0,1}$ is isomorphic to $\P^{2}$, and this is not a fine moduli space.

The moduli spaces $M^{\wF}_{0,n}$ with $n \in \{2,3,4\}$ have more complicated features. 
Since every rank $2$ weak Fano bundle $\mcE$ with $c_{1}(\mcE)=0$ and $c_{2}(\mcE) \in \{2,3,4\}$ is a slope stable instanton bundle
(see Remark~\ref{rem-stability} and \cite[Corollary~4.7]{FHI20}), 
$M^{\wF}_{0,n}$ is an open subscheme of the coarse moduli space $M^{\ins}_{0,n}$ of instanton bundles $\mcE$ on $X$ with $c_{1}(\mcE)=0$ and $c_{2}(\mcE)=n$ 
(for the definition of $M^{\ins}_{0,n}$ we refer to \cite{Faenzi14,Kuznetsov12}). 
When $n \in \{2,3\}$, it was proved by \cite{Sanna14,Sanna17} that $M^{\ins}_{0,n}$ is smooth, irreducible, but not projective. 
When $n=2$, it was known that $M^{\wF}_{0,2}=M^{\ins}_{0,2}$ since every minimal instanton bundle is a weak Fano bundle (c.f. \cite[Section~1.3]{FHI20}). 
Hence $M^{\wF}_{0,2}$ is not fine as a moduli space by \cite[Proposition~5.12]{Sanna17}. 
On the other hand, Sanna shows that $M^{\ins}_{0,3}$ is fine as a moduli space \cite[Remark 4.11]{Sanna17} and contains $M^{\wF}_{0,3}$ as a proper open subvariety. 
Hence $M^{\wF}_{0,3}$ is a fine moduli space. 

In this article, we use our description to establish the following properties of $M^{\wF}_{0,4}$. 
\begin{thm}\label{mainthm-moduli}
$M_{0,4}^{\wF}$ is an irreducible smooth quasi-projective variety of dimension $13$.  
Moreover, $M_{0,4}^{\wF}$ is not fine as a moduli space. 
In particular, the coarse moduli space $M^{\wF}_{0,c_{2}}$ is fine if and only if $c_{2} = 3$. 
\end{thm}

\subsubsection{Classifications on del Pezzo $3$-folds of degree $\leq 2$}

By Theorem~\ref{mainthm-resol} and our previous works \cite{Ishikawa16,FHI20}, 
we obtain a classification result on del Pezzo $3$-folds $X$ with $\deg X \in \{3,4,5\}$. 
To complete our classification, in the following theorem we show that there exists no non-trivial weak Fano bundles when $\deg X \leq 2$.

\begin{thm}\label{mainthm-Ishikawa}
Let $X$ be a del Pezzo $3$-fold of degree $d \leq 2$. 
Then every rank $2$ weak Fano bundle on $X$ is the direct sum of some line bundles. 
\end{thm}
Hence, combining our previous articles \cite{Ishikawa16,FHI20} with Theorem \ref{mainthm-resol} and \ref{mainthm-Ishikawa}, we finally obtain a classification of weak Fano bundles on del Pezzo $3$-folds of Picard rank $1$.


\subsection{Strategy for Proofs}\label{subsec-PlanOfProof}

Whilst the proof of Theorem~\ref{mainthm-Ishikawa} will employ a similar strategy used in \cite{Ishikawa16,FHI20},
new ideas are needed  for proving Theorems~\ref{mainthm-resol} and \ref{mainthm-moduli}. 
In this section, we explain a brief overview of these.

\subsubsection{Strategy for Proof of Theorem~\ref{mainthm-resol}}

Let $X$, $\mcR$, and $\mcQ$ be as in Section~\ref{subsec-intro-resol-dP5},
and $\mcE$ a weak Fano bundle on $X$. 
When $c_{1}(\mcE)=-1$, we have $\mcE|_{l} \simeq \mcO_{\P^{1}}(-1) \oplus \mcO_{\P^{1}}$ for every line $l$ on $X$. 
This allows us to characterize $\mcE$ by using the Hilbert scheme of lines on $X$ and its description which was given mainly by Furushima-Nakayama \cite{FN89} (see Section~\ref{sec-c1odd}).

When $c_{1}(\mcE)=0$, we study $\mcE$
for each possible value of its 2nd Chern class $c_{2}(\mcE)$. 
If $\mcE$ is indecomposable, then $c_{2}(\mcE) \geq 1$ (see Proposition~\ref{prop-Ishikawa}). 
When $c_{2}(\mcE)=1$, then the resolution in Theorem~\ref{mainthm-resol}~(v) is obtained by using the sequence $0 \to \mcO_{X} \to \mcE \to \mcI_{l/X} \to 0$. 
When $c_{2}(\mcE) \geq 2$, 
then $\mcE$ is an instanton bundle by \cite[Corollary~4.7]{FHI20}. Moreover, Faenzi \cite{Faenzi14} and Kuznetsov \cite{Kuznetsov12} gave a monad (= a self-dual three-term complex) for each instanton bundle, which yields Theorem~\ref{mainthm-resol}~(vi) for $c_{2}(\mcE) = 2$. 
Thus, the non-trivial part of the proof of Theorem~\ref{mainthm-resol} is for $c_{1}(\mcE)=0$ and $c_{2}(\mcE) \geq 3$.

To study such an $\mcE$, we use the bounded derived category $\Db(X)$ of coherent sheaves on $X$, which admits a full exceptional collection $\Db(X) = \langle \mcO_{X}(-1),\mcQ(-1),\mcR,\mcO_{X} \rangle$ \cite{Orlov91}. 
By our previous work \cite[Theorem~1.7]{FHI20}, $\mcE(1)$ is globally generated,
and thus there is an exact sequence $0 \to \mcK \to H^{0}(\mcE(1)) \otimes \mcO_{X} \to \mcE(1) \to 0$. 
Together with the Kawamata-Viehweg vanishing, this exact sequence shows that the left mutation $\LL_{\mcO_{X}}(\mcE(1))$ is concentrated in the single degree $-1$. 
As $\LL_{\mcO_{X}}(\mcE(1)) \in \braket{\mcO_{X}(-1),\mcQ(-1),\mcR}$, 
the right mutation $\RR_{\mcO_{X}(-1)}\LL_{\mcO_{X}}(\mcE(1))$ belongs to $\braket{\mcQ(-1),\mcR}$.
By showing some cohomology groups associated with $\mcE$ vanish, 
it will turn out that there exists a vector bundle $\mcV$ such that  $\mcV \simeq \RR_{\mcO_{X}(-1)}\LL_{\mcO_{X}}(\mcE(1))[-1]$ and $\mcV$ lies in an exact sequence 
\[ 0 \to \mcO_{X}(-1) \otimes H^{1}(\mcE) \to \mcV \to \mcO_{X} \otimes H^{0}(\mcE(1)) \to \mcE(1) \to 0. \]
Thus the problem comes down to giving a concrete description  of the vector bundle $\mcV$.
The fact that $\mcV$ is contained in the category $\braket{\mcQ(-1),\mcR}$ generated by the two-term exceptional collection gives an exact sequence in Lemma~\ref{lem-conclu},
which brings the problem down to the calculation of $\Ext^{\bullet}(\mcR,\mcE)$ and $\Ext^{\bullet}(\mcQ(-1),\mcE)$.

The most technical part of the proof of Theorem~\ref{mainthm-resol} is to compute $\Ext^{\bullet}(\mcR,\mcE)$ and $\Ext^{\bullet}(\mcQ(-1),\mcE)$. 
The computation of these $\Ext$-groups is more complicated than it appears, and would not be deduced from stability or instatonic property of $\mcE$.
Indeed, we will see in Proposition~\ref{prop-chara} that the vanishing $\Ext^1(\mcQ(-1), \mcE) = 0$ is a strong condition that characterizes weak Fano bundles among all instanton bundles.
From the vanishing of $\Ext^{1}(\mcQ(-1),\mcE)$, it is easy to show that $\mcE$ is weak Fano, but the proof of the converse essentially requires an inequality that follows from the positivity of nef bundles (Lemma~\ref{lem-codim2class}). This suggests that the above $\Ext$-group reflects the positivity of $\mcE(1)$.

Our approach using exceptional collections is inspired by Ohno's work \cite{Ohno16v7}. 
In that paper, he classified nef vector bundles on the projective space using a full strong exceptional collection of its derived category and  the spectral sequence, which is called the Bondal spectral sequence. 
In contrast, we use mutations over exceptional collections and Serre functors for admissible subcategories instead of the Bondal spectral sequence.


\subsubsection{Plan of Proof of Theorem~\ref{mainthm-moduli}}

To prove Theorem~\ref{mainthm-moduli}, 
we will use the resolution in Theorem~\ref{mainthm-resol}~(viii). 
Using this resolution allows us to construct an open embedding 
$M^{\wF}_{0,4} \to M^{\Theta\text{-st}}_{(2,2)}(Q)$, 
where 
$\Theta$ is a certain stability
and 
$M^{\Theta\text{-st}}_{(2,2)}(Q)$ is the moduli space of 
$\Theta$-stable representations of 
the $5$-Kronecker quiver $Q$ with dimension vector $(2,2)$. 
Thus the irreducibility of $M_{0,4}^{\wF}$ follows from that of $M^{\Theta\text{-st}}_{(2,2)}(Q)$. 
Moreover, the observation of natural Brauer-Severi schemes on $M^{\Theta\text{-st}}_{(2,2)}(Q)$ by \cite{RS17} immediately shows that $M^{\wF}_{0,4}$ cannot carry the universal bundle, i.e., $M^{\wF}_{0,4}$ is not fine. 

The moduli space $M^{\Theta\text{-sst}}_{(2,2)}(Q)$ of semi-stable representations has the following geometric feature. Chung-Moon \cite{Chung-Moon17} showed that $M^{\Theta\text{-sst}}_{(2,2)}(Q)$ is isomorphic to a variety $T_{4}$ that was introduced by Hosono-Takagi \cite{HT15}. 
This space $T_{4}$ is given as the Stein factorization of $U_{4} \to S_{4}$, where $S_{4} \subset |\mcO_{\P^{4}}(2)|$ is the closed subscheme that parametrizes the singular hyperquadrics in $\P^{4}$, and $U_{4}:=\{([\Pi],Q) \in \Gr(2,5) \times S_{4} \mid \Pi \subset Q\}$ is the incidence variety of $2$-planes and $S_{4}$. 
Actually, we can also conclude that there is no universal bundle on $M_{0,4}^{\wF}  \times X$ by seeing this fibration structure $U_{4} \to T_{4}$. 
In both way, the essential part is that the non-vanishing of certain Brauer classes on $T_{4}$ is namely an obstruction for the existence of the universal bundle. 

\subsection{Organization of this article}
In Section~\ref{sec-prelim}, we prepare our notation and background on bounded derived categories of coherent sheaves. 
In Section~\ref{sec-c1odd}, we study weak Fano bundles $\mcE$ on a del Pezzo $3$-fold of degree $5$ with $c_{1}(\mcE)=-1$. 
In Section~\ref{sec-resol-dP5} we study those with $c_{1}(\mcE)=0$, which is the key part of this article. 
In Section~\ref{sec-moduli}, we apply results obtained in Section~\ref{sec-resol-dP5} to study moduli spaces of rank $2$ weak Fano bundles on a del Pezzo $3$-fold of degree $5$, and prove Theorem~\ref{mainthm-moduli}. 

\begin{NaC}
Throughout this article, we will work over an algebraically closed field $\Bbbk$ of characteristic zero. 
We also adopt the following convention. 
\begin{itemize}
\item We regard vector bundles as locally free sheaves. 
For a vector bundle $\mcE$ on a smooth projective variety $X$, 
we define $\P_{X}(\mcE):=\Proj \Sym \mcE$. 
\item In this article, a \emph{del Pezzo $3$-fold} is a smooth Fano $3$-fold $X$ whose Fano index $r(X)=\min \{r \mid -K_{X} \sim rH \text{ for some } H \in \Pic(X)\} $ is $2$. 
For a del Pezzo $3$-fold $X$, $\mcO_{X}(1) \in \Pic(X)$ denotes the polarization such that $\mcO_{X}(2) \simeq \mcO_{X}(-K_{X})$. 
\item For the Grassmannian varieties, we employ the classical notation as follows; for an $n$-dimensional vector space $V=\Bbbk^{n}$ and a positive integer $m<n$, 
we define the Grassmannian variety $\Gr(m,V)=\Gr(m,n)$ as the parameter space of $m$-dimensional linear subspaces of $V=\Bbbk^{n}$. 
In particular, if we regard $V$ as a locally free sheaf on $\Spec \Bbbk$, 
then $\P(V)$ is canonically isomorphic to $\Gr(n-1,V)$ in our notation. 
\item The derived global section and the derived Hom are denoted by $\RG$ and $\RHom$, respectively.
For the definitions of these functors and other foundations of derived categories, we refer to \cite{Huybrechts06}.
\end{itemize}
\end{NaC}

\begin{ACK}
The first author is grateful to Professor Yoshinori Gongyo for his warm encouragement. 
He is also grateful to Doctor Naoki Koseki for answering his questions about moduli spaces and the representability of certain functors. 
The authors thank the anonymous referee for their careful reading and a lot of advice to improve the expositions.
The authors are grateful to Professor Masahiro Ohno for pointing out our errors in Proposition~\ref{prop-moduli-prelim}~(2) in the previous version of this paper and kindly discussing with us. The second author was supported by EP/R034826/1. 
\end{ACK}
 
\section{Preliminaries}

\subsection{Weak Fano bundles on del Pezzo threefolds of Picard rank one}

Let $X$ be a del Pezzo $3$-fold of Picard rank $1$ and $\mcE$ a vector bundle on $X$. 
Let $\pi \colon \P_{X}(\mcE) \to X$ be the projectivization and $\xi_{\mcE}$ be a tautological divisor. 
As in section~1, $\mcE$ is called a \emph{weak Fano bundle} if $-K_{\P_{X}(\mcE)}$ is nef and big \cite{Langer98}. 
We also recall the notion of the nefness and the ampleness of the vector bundles; $\mcE$ is called a \emph{nef} (resp. \emph{ample}) vector bundle if $\xi_{\mcE}$ is nef (resp. ample) \cite{Laz}. 

In this article, we often identify $N^{i}(X)_{\Z} \simeq \Z$ by taking the effective generator class for each $i \in \{0,1,2,3\}$, 
where $N^{i}(X)_{\Z}$ is the numerical class group of the codimension $i$ cycles with $\Z$-coefficients. 
In this article, the $i$-th Chern class $c_{i}(\mcE)$ is often regarded as a numerical class. 
We also say that a rank $2$ vector bundle $\mcE$ on a $X$ is \emph{normalized} when $c_{1}(\mcE) \in \{0,-1\}$. 
\begin{rem}\label{rem-amp}
If $\mcE$ is of rank $2$, then $-K_{\P(\mcE)} \sim 2\xi_{\mcE}-\pi^{\ast}(K_{X}+\det \mcE)$, where $\mcO_{X}(-K_{X}) \simeq \mcO_{X}(2)$ on $X$. 
Hence if $\mcE$ is a normalized weak Fano bundle of rank $2$, then $\mcE(2)$ is ample. 
Moreover, $\mcE(1)$ is nef if $c_{1}(\mcE)=0$. 
\end{rem}
\begin{rem}\label{rem-st}
We also follow the notion of \emph{slope stability} and \emph{Gieseker stability} defined in \cite{OSS80,HL10}. 
If $\mcE$ is of rank $2$ and normalized, the following conditions are equivalent.
\begin{itemize}
\item $\mcE$ is slope stable. 
\item $\mcE$ is Gieseker-stable.
\item $\mcE$ is simple. 
\item $h^{0}(\mcE)=0$. 
\end{itemize}
For proofs, see \cite[Lemma~1.2.13]{HL10}, \cite[Corollary~1.2.8]{HL10}, \cite[Ch.~2, Theorem~1.2.10 and Lemma 1.2.5]{OSS80} and the proofs of them for example. 
\end{rem}



The following proposition summarises fundamental properties of rank $2$ weak Fano bundles on del Pezzo $3$-folds. 
Recall that a \textit{line} is an irreducible curve $l \subset X$ such that $\deg \mcO(1)|_l = 1$.

\begin{prop}\label{prop-Ishikawa}
Let $X$ be a del Pezzo $3$-fold of degree $d \leq 5$. 
Let $\mcE$ be a normalized rank $2$ weak Fano vector bundle on $X$. 
\begin{enumerate}
\item $\chi(\mcE(n))=\frac{d}{3}n^{3}+\frac{(c_{1}+2)d}{2}n^{2}+\frac{(3c_{1}+4)d-6c_{2}+12}{6}n+\frac{1}{2}(c_{1}+2)(2-c_{2})$.
\item The following vanishing of cohomology hold.
\[
H^{i}(\mcE(n)) = 0
\begin{cases}
\text{if $n \geq 0$ and 
$i \geq \max\{2-\frac{n}{c_{1}^{2}+1},1\}$, or} \\
\text{if $n \leq (-2-c_{1})$ and 
$i \leq \min\{2-\frac{n+3}{c_{1}^{2}+1},2\}$.}
\end{cases}
\]
\item $\frac{4+c_{1}^{2}}{4}d > c_{2}$. 
\item $c_{2}(\mcE) \leq 0$ if and only if $\mcE$ is a direct sum of line bundles. 
\item The following are equivalent. 
	\begin{enumerate}
	\item $c_{2}(\mcE) = 1$. 
	\item $c_{2}(\mcE)>0$ and $h^{0}(\mcE) \neq 0$. 
	\item $d=\deg X \geq 3$ and there is an exact sequence $0 \to \mcO_{X} \to \mcE \to \mcI_{l/X} \to 0$, where $l$ is a line on $X$.
	\end{enumerate}
	In particular, if $\deg X \leq 2$, then $c_{2}(\mcE) \neq 1$. 
\item  $c_{2}(\mcE) \geq 2$ if and only if $h^{0}(\mcE)=0$. 
\end{enumerate}
\end{prop}
\begin{proof}
Let $c_{i} \in \Z$ be the integer such that $\det \mcE=\mcO_{X}(-c_{1})$ and $c_{2}(\mcE)=c_{2} \cdot l \in H^{4}(X,\Z)$. Then $c_{1} \in \{-1,0\}$.

(1) follows from the Hirzebruch-Riemann-Roch theorem. 

(2) 
First, recall that $\mcE(2)$ is ample by Remark \ref{rem-amp}.
Thus 
the Le Potier vanishing \cite[Theorem 7.3.5]{Laz} yields that $H^i(X, \mcE(n)) = 0$ for all $i \geq 2$ and $n \geq 0$.
In addition,  the Kawamata-Viehweg vanishing (applied to $\mcO_{\P(\mcE)}(1) \otimes \pi^{\ast}\mcO_{X}(n)$ on $\P(\mcE)$ together with the projection formula)
implies that $H^1(X, \mcE(n)) = 0$ for all $n \geq 2$.
Furthermore, if $c_1(\mcE) = 0$, since $\mcE(1)$ is nef, the Kawamata-Viehweg vanishing also gives $H^1(X, \mcE(n)) = 0$ for all $n \geq 1$.
Summarising all the above gives the first statement.
The second statement follows from the first statement.
Indeed, since $\mcE$ is rank two, $\mcE^{\vee} \simeq \mcE(-c_1)$.
Thus the Serre duality gives an isomorphism $H^i(X, \mcE(n)) \simeq H^{3-i}(X, \mcE(-n - 2-c_1))^*$.
Using this allows us to check the second vanishing.

(3) follows from the inequality $0<(-K_{\P(\mcE)})^{4}=16(c_{1}^{2}d+4d-4c_{2})$. 

For (4), combining (1) and (2) shows that 
\begin{align}\label{eq-beta}
h^{0}(\mcE(-2))=0
\end{align}
and 
\begin{align}\label{eq-HRR}
h^{0}(\mcE) \geq \chi(\mcE)=\frac{1}{2}(c_{1}+2)(2-c_{2})
\end{align}
for every normalized weak Fano bundle $\mcE$ of rank $2$. 
Thus it suffices to see that every normalized weak Fano bundle $\mcE$ of rank $2$ with $c_{2} \leq 0$ decomposes into a direct sum of line bundles. 

If $h^{0}(\mcE(-1)) \neq 0$, then there is an exact sequence $0 \to \mcO_{X}(1) \to \mcE \to \mcI_{Z/X}(-1+c_{1}) \to 0$, where $Z$ is a closed subscheme with $\codim_{X}Z \geq 2$ by (\ref{eq-beta}). 
If $c_{1}=-1$, then $\mcE(2)$ is ample by Remark~\ref{rem-amp} and there is a surjection $\mcE(2) \epm \mcI_{Z/X}$, which is a contradiction. 
Hence $c_{1}=0$. 
Then $\mcE(1)$ is nef by Remark~\ref{rem-amp} and there is a surjection $\mcE(1) \epm \mcI_{Z/X}$, which implies $Z=\emp$. 
Hence it holds that $\mcE \simeq \mcO_{X}(1) \oplus \mcO_{X}(-1)$. 
Suppose $h^{0}(\mcE(-1))=0$. 
By (\ref{eq-HRR}) and the condition $c_{2} \leq 0$, $h^{0}(\mcE)>0$ and hence there is an exact sequence 
$0 \to \mcO_{X} \to \mcE \to \mcI_{Z/X}(c_{1}) \to 0$, 
where $Z$ is a closed subscheme with $\codim_{X}Z \geq 2$. 
Then $Z \equiv c_{2}(\mcE)$ implies that $Z$ is empty since $c_{2} \leq 0$. 
Hence $\mcE \simeq \mcO_{X} \oplus \mcO_{X}(c_{1})$. 

(5) The implication (a) $\ra$ (b) follows from (1) and (2). 
We show (b) $\ra$ (c). 
If $c_{2}(\mcE)>0$ and $h^{0}(\mcE) \neq 0$, then by (4) and its proof, 
$\mcE$ is indecomposable and $h^{0}(\mcE(-1))=0$. 
As discussed in (4), there is an exact sequence 
$0 \to \mcO_{X} \to \mcE \to \mcI_{Z/X}(c_{1}) \to 0$ with $\codim_{X}Z \geq 2$. 
Let $\pi_{\mcE} \colon \P_{X}(\mcE) \to X$ be the projectivization of $\mcE$ and $\xi_{\mcE}$ be a tautological divisor. 
Since $h^{0}(\mcE)>0$, $\xi_{\mcE}$ is effective, 
which implies 
$0 \leq (-K_{\P_{X}(\mcE)})^{3}\xi_{\mcE} =  
(c_{1}^{3}+6c_{1}^{2}+12c_{1}+8)d-(4c_{1}c_{2}+24c_{2})$.
Since $c_{2} > 0$ and $c_{1} \in \{0,-1\}$, it follows that $c_{1}=0$, $c_{2}=1$, and $d \geq 3$. 
Then $Z$ is of degree $c_{2}=1$, which implies $Z$ is a line. 
Hence (b) $\ra$ (c) holds. 
The implication (c) $\ra$ (a) is obvious.

(6) follows from (5) and (\ref{eq-HRR}). 
We complete the proof. 
\end{proof}

\subsection{Derived categories}\label{sec-prelim}

This section recalls some notations from triangulated and derived categories. 
The triangulated categories we are mainly interested in are the bounded derived categories of coherent sheaves on smooth projective varieties $X$, denoted by $\Db(X)$. 
We adopt the terminology of \cite{Huybrechts06}. 

Let $\mcD$ be an $\Ext$-finite $\Bbbk$-linear triangulated category. 
For objects $A,B \in \mcD$, we denote $\bigoplus_{i} \Hom_{\mcD}(A,B[i])[-i]$ by $\Ext^{\bullet}(A,B)$. 

\begin{defi}
\begin{enumerate}
\item An object $E \in \mcD$ is \textit{exceptional} if $\Hom(E, E) \simeq \Bbbk$ and $\Ext^i(E, E) =0$ for all $i \neq 0$.
\item A sequence of exceptional objects $E_1, E_2, \cdots, E_l$ is called an \textit{exceptional collection} if $\Ext^k(E_i, E_j) = 0$ for all $1 \leq j < i \leq l$ and $k \in \mathbb{Z}$.
\item A \textit{full} exceptional collection is an exceptional collection $E_1, E_2, \cdots, E_l$ that generates the whole category $\mcD$. Here generate means the smallest triangulated full subcategory that contains $E_1, E_2, \cdots, E_l$ agrees with $\mcD$.
\item A \textit{strong} exceptional collection is an exceptional collection $E_1, E_2, \cdots, E_l$ that in addition satisfies $\Ext^k(E_i, E_j) = $ for all $0 \leq i < j \leq l$ and $k \neq 0$.
\end{enumerate}
\end{defi}

An explicit example of a full strong exceptional collection will appear in Theorem \ref{Orlov}.

In this article, we define mutations of objects over exceptional objects as follows. 
\begin{defi}
Let $E$ be an exceptional object of $\mcD$ and $F$ an object of $\mcD$. 
\begin{enumerate}
\item The \emph{left mutation} $\LL_{E}(F)$ of $F$ over $E$ is defined to be the cone of the natural evaluation morphism $\ev \colon \Ext^{\bullet}(E, F) \otimes E \to F$: 
\[ \Ext^{\bullet}(E, F) \otimes E \xrightarrow{\mathrm{ev}} F \to \LL_E(F) \mathop{\to}^{+1} . \]
\item The \emph{right mutation} $\RR_{E}(F)$ of $F$ over $E$ is defined to be the cone of $\coev[-1] \colon F \to \Ext^{\bullet}(F,E)^{\vee} \otimes E[-1]$, where $\coev$ is the natural co-evaluation morphism: 
 \[ \RR_E(F) \to F \xrightarrow{\mathrm{coev}} \Ext^{\bullet}(F, E)^{\vee} \otimes E\mathop{\to}^{+1}. \]
\end{enumerate}
For a full subcategory $\mcA \subset \mcD$, we define
\begin{align*}
\langle \mcA \rangle^{\bot} &= \{ F \in \mcD \mid \Ext^{\bullet}(A, F) = 0 \text{ for all object $A \in \mcA$} \},  \\
{}^{\bot}\langle \mcA \rangle &= \{ F \in \mcD \mid \Ext^{\bullet}(F, A) = 0 \text{ for all object $A \in \mcA$} \}.
\end{align*}
\end{defi}

\begin{rem}\label{rem-cat}
If $\mcD = \braket{E_1, \cdots, E_r}$ is a full exceptional collection, then for any $1 \leq j \leq k \leq r$, 
\[ {}^{\bot} \braket{E_1, \cdots, E_{j-1}} \cap \braket{E_{k+1}, \cdots, E_r}^{\bot} = \braket{E_j ,\cdots, E_k}. \]
Furthermore, we collect the following well-known facts for later convenience. 
\begin{enumerate}
\item Let $E \in \mcD$ be an exceptional object. 
Then the mutations over $E$ define functors $\LL_{E} \colon \mcD \to \braket{E}^{\bot}$ and $\RR_{E} \colon \mcD \to {}^{\bot}\braket{E}$. 
It is easy to see that the left mutation $\LL_E$ (resp. the right mutation $\RR_E$) over $E$ is the left (resp. right) adjoint functor of the inclusion $\langle E \rangle^{\bot} \hookrightarrow \mcD$ (resp. ${}^{\bot}\langle E \rangle \hookrightarrow \mcD$).
\item Mutations act on exceptional collections as follows. 
For a given exceptional collection 
$(E_{1},\ldots,E_{i},E_{i+1},\ldots,E_{r})$, 
the sequences 
$(E_{1},\ldots,E_{i-1},\LL_{E_{i}}(E_{i+1}),E_{i},E_{i+2},\ldots,E_{r})$ 
and 
$(E_{1},\ldots, E_{i-1},E_{i+1},\RR_{E_{i+1}}(E_{i}), E_{i+2},\ldots,E_{r})$
are also exceptional collections. 
These three collections generate the same category. 
\item Related to (2), we observe that Serre functors also act on exceptional collections. 
If $\mcD=\braket{E_{1},E_{2},\ldots,E_{r}}$ is a full exceptional collection and there is a Serre functor $S_{\mcD}$ of $\mcD$, 
then $\braket{S_{\mcD}(E_r), E_1, \cdots, E_{r-1}}$ and $\braket{E_2, \cdots, E_r, S_{\mcD}^{-1}(E_1)}$ are also full exceptional collections of $\mcD$. 
\end{enumerate}
\end{rem}

We also prepare the following easy lemma. 
\begin{lem}\label{lem-2term}
Let $(E_{1},E_{2})$ be an exceptional collection and $F \in \braket{E_{1},E_{2}}$ be an object. 
Then 
$\RR_{E_{1}}(F) \simeq \Ext^{\bullet}(E_{2},F) \otimes E_{2}$. 
In particular, there is a distinguished triangle
\[\Ext^{\bullet}(E_{2},F) \otimes E_{2} \to F \to \Ext^{\bullet}(F,E_{1})^{\vee} \otimes E_{1} \mathop{\to}^{+1}. \]
\end{lem}
\begin{proof}
It follows from Remark~\ref{rem-cat}~(1) that $\RR_{E_{1}}(F) \in {}^{\bot}\braket{E_{1}}=\braket{E_{2}}$. 
Since $E_{2}$ is an exceptional object, there exist some finite dimensional vector spaces $V_{i}$ such that 
$\RR_{E_{1}}(F) \simeq E_{2} \otimes \bigoplus_{i} V_{i}[i]$. 
It follows from the definition of the mutation that 
$\Ext^{\bullet}(E_{2},\RR_{E_{1}}(F)) \simeq \Ext^{\bullet}(E_{2},F)$. 
Since the left hand side of the above isomorphism is $\bigoplus_{i} V_{i}[i]$, 
$\RR_{E_{1}}(F) \simeq \Ext^{\bullet}(E_{2},F) \otimes E_{2}$. 
\end{proof}

Let $X$ be  a del Pezzo $3$-fold of degree $5$.
It is known that $X$ is isomorphic to a codimension $3$ linear section of $\Gr(2,5) \subset \mathbb{P}^9$.
Restricting the rank two universal subbundle (resp.~ rank three universal quotient bundle) on $\Gr(2,5)$ to $X$ gives the bundle $\mcR$ (resp.~$\mcQ$).
By construction, these bundles fits into the tautological exact sequence
\[ 0 \to \mcR \to \mcO_X^{\oplus 5} \to \mcQ \to 0. \]
Using them Orlov provided the following full strong exceptional collection, which is one of the key ingredients for our investigation of rank $2$ weak Fano vector bundles on $X$.

\begin{thm}\label{Orlov}
The derived category $\Db(X)$ of a del Pezzo $3$-fold of degree $5$ admits a full strong exceptional collection
\begin{align}\label{eq-Orlov}
\Db(X) = \langle \mcO_{X}(-1), \mcQ(-1), \mcR, \mcO_{X} \rangle.
\end{align}
\end{thm}

\begin{rem}
This full strong exceptional collection can be seen as the exceptional collection $\e_{2}$ in \cite[Theorem]{Orlov91} tensored by $\mcO_{X}(-1)$. 
\end{rem}

\begin{rem}
Note that by the definition of a full strong exceptional collection,
the statement of this theorem contains a lot of vanishings, like $\Ext^{\geq 1}(\mcR, \mcO_X) = 0$ and $\RHom(\mcR, \mcQ(-1)) = 0$, etc.
In later sections we freely use these vanishings.

In addition, the tautological sequence together with those vanishings implies $\RG(\mcQ) \simeq \Bbbk^{\oplus 5}$. This will be used in Lemma \ref{lem-vani-c24}.
\end{rem}

We also note that the bundles $\mcR$ and $\mcQ$ are slope stable since $H^0(\mcR) = 0$ and $H^0(\mcQ(-1)) = 0 = H^0(\mcQ^{\vee})$.
In addition, the Chern classes of $\mcR$ are given by $c_1(\mcR) = -1$ and $c_2(\mcR) = 2$, and those of $\mcQ$ are given by $c_1(\mcQ) = 1$, $c_2(\mcQ) = 3$, and $c_3(\mcQ) = 1$.
Thus the Chern character of $\mcQ$ is given by 
\[ \ch(\mcQ) =  \left(3, H_X, -\frac{1}{2}l, -\frac{1}{6} \right). \]
This will be used in the proof of Lemma~\ref{lem-vani-c24}.

\section{Rank $2$ weak Fano bundles with odd first Chern classes}\label{sec-c1odd}
Throughout Sections~\ref{sec-c1odd} and \ref{sec-resol-dP5}, 
Let $X$ denote a del Pezzo $3$-fold of degree $5$. 
We regard $X \subset \mathbb{P}^6$ as a codimension $3$ linear section of $\Gr(2,5)$ under the Pl\"{u}cker embedding. 
Let $\mcR$ and $\mcQ$ be the restriction of the universal rank $2$ subbundle and rank $3$ quotient bundle respectively. 

In this section, we deal with such an $\mcE$ with $c_{1}(\mcE)=-1$ by observing the Hilbert scheme of lines on a del Pezzo $3$-fold $X$ of degree $5$. 
The Hilbert scheme of lines on $X$ is well studied by several researchers, including the authors of \cite{FN89}, \cite{Kuznetsov12}. 
Since the ample generator $H$ gives the above embedding $X \subset \P^6$, a line on X is a $1$-dimensional linear subspace in $\P^6$ that is contained in $X$.
We use the following result, which is a somewhat rough statement compared to \cite{FN89}. 
\begin{thm}[{\cite[Theorem~I, Lemma~2.3, Lemma~2.1]{FN89}}]
\label{thm-FN}
There exists the morphisms
\[X \mathop{\gets}^{e} U = \P_{\P^{2}}(\mcG) \mathop{\to}^{\pi} \P^{2},\]
where 
\begin{itemize}
\item $\pi \colon \P_{\P^{2}}(\mcG)=U \to \P^{2}$ is the universal family of the lines on $X$, which is the projectivization of a rank $2$ vector bundle $\mcG$, 
\item $e \colon U \to X$ is the evaluation morphism, which is a finite morphism of degree $3$, and 
\item $\mcG$ is a rank $2$ globally generated vector bundle with $c_{1}(\mcG)=3$ and $c_{2}(\mcG)=6$. 
\end{itemize}
Moreover, if $\eta$ denotes a tautological divisor of $\P_{\P^{2}}(\mcG)$ and $L$ the pull-back of a line on $\P^{2}$ under 
$\pi \colon \P_{\P^{2}}(\mcG) \to \P^{2}$, 
then 
\begin{enumerate}
\item $e^{\ast}H_{X} \sim \eta+L$ for a hyperplane section $H_{X}$ on $X$, and 
\item $e^{\ast}l \sim \eta.L-L^{2}$ for a line $l$ on $X$. 
\end{enumerate}
\end{thm}
We refer to \cite{FN89} for the proof of Theorem~\ref{thm-FN} and more details. 
Using Theorem~\ref{thm-FN}, 
we show the following proposition. 
\begin{prop}\label{prop-c1odd}
Let $\mcE$ be a rank $2$ bundle on $X$ with $c_{1}(\mcE)=-1$. 
Assume that $\mcE|_{l} \simeq \mcO_{l} \oplus \mcO_{l}(-1)$ for every line $l$ on $X$. 
Then $\mcE$ is $\mcO_{X} \oplus \mcO_{X}(-1)$ or $\mcR$. 
\end{prop}
\begin{proof}
Set $\mcE_{U}:=e^{\ast}\mcE$. 
Then $\mcE_{U}|_{\pi^{-1}(p)} \simeq \mcO_{\P^{1}} \oplus \mcO_{\P^{1}}(-1)$ for every $p \in \P^{2}$. 
Hence there is an integer $a \in \Z$ such that $\mcO_{\P^2}(-a) \simeq \pi_{\ast}\mcE_{U}$. 
Note that the cokernel of the natural injection $\pi^{\ast}\mcO_{\P^{2}}(-a) \to \mcE_{U}$ is invertible. 
Since $\det \mcE_{U} \simeq e^{\ast}\mcO_{X}(-1) \simeq \mcO_{U}(-(\eta+L))$, there exists the following exact sequence: 
\begin{align}\label{ex-split}
0 \to \mcO_{U}(-aL) \to \mcE_{U}=e^{\ast}\mcE \to \mcO_{U}(-\eta+(a-1)L) \to 0. 
\end{align}
Thus $c_{2}(\mcE_{U})=(-aL)(-\eta+(a-1)L)=a\eta L-a(a-1)L^{2}$. 
On the other hand, by Theorem~\ref{thm-FN}~(2), $c_{2}(\mcE_{U})=e^{\ast}c_{2}(\mcE)$ must be divisible by $\eta.L-L^{2}$ in $N^{2}(U)_{\Z}$. 
Hence $a \in \{0,2\}$ and $c_{2}(\mcE)=a$. 

When $a=c_{2}(\mcE)=0$, 
the exact sequence (\ref{ex-split}) gives $\RG(U, e^*\mcE) \simeq \Bbbk$ and $H^0(U, e^*\mcE(-e^*H)) = 0$.
Since $e$ is a finite morphism, $Re_* \simeq e_*$ holds and $e_*e^*\mcE$ contains $\mcE$ as a direct summand.
Thus $H^i(\mcE) = 0$ for all $i > 0$, and $H^0(\mcE(-1)) = 0$.

Now, by the Hirzebruch-Riemann-Roch theorem, $h^0(\mcE) = \chi(\mcE) = 1$, and let $s \colon \mcO_X \to \mcE$ be the non-zero section.
Since $H^{0}(\mcE(-1))=0$, there exists a closed subscheme $Z \subset X$ with $\codim_X Z \geq 2$ such that $\Cok(s) \simeq \mcI_{Z/X}(-1)$.
The equality $Z \equiv c_2(\mcE) = 0$ yields that $Z$ is empty and hence $\mcE$ fits in the sequence
\[ 0 \to \mcO_X \to \mcE \to \mcO_X(-1) \to 0. \]
Thus the vanishing $\Ext_X^1(\mcO_X(-1), \mcO_X) = 0$ implies $\mcE \simeq \mcO_X \oplus \mcO_X(-1)$.

When $a=c_{2}(\mcE)=2$, 
the exact sequence (\ref{ex-split}) shows $\RG(U,e^{\ast}\mcE)=0$. 
Since $e_{\ast}e^{\ast}\mcE$ has $\mcE$ as a direct summand, 
it holds that $\RG(X,\mcE)=0$. 
Note that $\mcE \simeq \mcE^{\vee}(-1)$ since $\mcE$ is of rank $2$ and $\mcE \otimes \mcE \to \bigwedge^{2} \mcE \simeq \mcO_{X}(-1)$ is a perfect pairing.
Hence
$\RHom(\mcE,\mcO_{X}(-1)) = \RG(\mcE^{\vee}(-1))=\RG(\mcE)=0$.
Hence $\mcE \in \braket{\mcO_{X}}^{\bot} \cap {}^{\bot}\braket{\mcO_{X}(-1)}$ in $\Db(X)$, which means that $\mcE \in \braket{\mcQ(-1),\mcR}$. 
Then Lemma~\ref{lem-2term} gives the following distinguished triangle:
\[\RHom(\mcR,\mcE) \otimes \mcR \to \mcE \to \RHom(\mcE,\mcQ(-1))^{\vee} \otimes \mcQ(-1) \mathop{\to}^{+1}.\]
Note that $\mcE$ is slope stable with slope $\mu(\mcE) = - 1/2$, since $h^{0}(\mcE)=0$ (see Remark~\ref{rem-st}). 
Similarly, $\mcQ$ is slope stable with slope $\mcQ(-1) = -2/3 < -1/2$.
Thus $\Hom(\mcE,\mcQ(-1))=0$. 
Taking the long exact sequence from the above distinguished triangle, we have 
\[0 \to \mcQ(-1)^{\oplus a} \to \mcR^{\oplus b} \to \mcE \to 0,\]
where $a=\ext^{1}(\mcE,\mcQ(-1))$ and $b=\hom(\mcR,\mcE)$. 
Computing the rank and the 1st Chern class, 
we have $3a+2=2b$ and $-2a-1=-b$, 
which implies $a=0$ and $b=1$. 
Thus $\mcE \simeq \mcR$ as claimed.  
\end{proof}

\begin{cor}\label{cor-c1odd}
Let $X$ be a del Pezzo $3$-fold of degree $5$. 
Then every rank $2$ weak Fano bundle $\mcE$ on $X$ with $c_{1}(\mcE)=-1$ 
is isomorphic to $\mcO_{X} \oplus \mcO_{X}(-1)$ or $\mcR$. 
\end{cor}
\begin{proof}
By Remark~\ref{rem-amp}, $\mcE(2)$ is ample. 
Hence $\mcE(2)|_{l} \simeq \mcO_{\P^{1}}(1) \oplus \mcO_{\P^{1}}(2)$ for every line $l \subset X$. 
Then Proposition~\ref{prop-c1odd} implies the result.
\end{proof}


\section{Rank $2$ weak Fano bundles with $c_{1}=0$}\label{sec-resol-dP5}
In this section, we deal with a rank $2$ weak Fano bundle $\mcE$ with $c_{1}(\mcE)=0$ to complete the proof of Theorem~\ref{mainthm-resol}. 
The main purpose of this section is to show the following theorem, which is an essential part of Theorem~\ref{mainthm-resol}. 
\begin{thm}\label{thm-resol}
Let $\mcE$ be an indecomposable rank $2$ weak Fano bundle on $X$ with $c_{1}(\mcE)=0$. 
If $c_{2}(\mcE)=1,2,3,4$, then $\mcE$ satisfies (v), (vi), (vii), (viii) in Theorem~\ref{mainthm-resol} respectively. 
\end{thm}
To show the above theorem, we mainly study a rank $2$ weak Fano bundle $\mcE$ with $c_{1}(\mcE)=0$ in this section. 
The following proposition is a fundamental property of such bundles. 
\begin{prop}\label{prop-FHI}
Let $\mcE$ be a rank $2$ vector bundle on a del Pezzo $3$-fold $X$ of degree $5$ with $c_{1}(\mcE)=0$. 
\begin{enumerate}
\item $\mcE$ is weak Fano bundle if and only if $\mcE(1)$ is globally generated and $c_{2}(\mcE) < 5$. 
\item If $\mcE$ is weak Fano with $c_{2}(\mcE) \geq 2$, then $\mcE$ is an instanton bundle. In particular, $\RG(\mcE(-1))=0$. 
\end{enumerate}
\end{prop}
\begin{proof}
Set $\mcF:=\mcE(1)$. 
Then $-K_{\P(\mcF)}=2\xi_{\mcF}$, where $\xi_{\mcF}$ is a tautological divisor on $\P(\mcF)$.  
Hence $\mcE$ is weak Fano if and only if $\mcF$ is nef and $0<\xi_{\mcF}^{4}=c_{1}(\mcF)^{3}-2c_{1}(\mcF)c_{2}(\mcF)$.
Note that the inequality $c_{2}(\mcE)<5$ is equivalent to the inequality $0<c_{1}(\mcF)^{3}-2c_{1}(\mcF)c_{2}(\mcF)=40-4(c_{2}(\mcE)+5)=20-4c_{2}(\mcE)$. 

Thus, if $\mcF = \mcE(1)$ is globally generated and $c_2(\mcE) < 5$, $\mcE$ is weak Fano.
Conversely, if $\mcE$ is weak Fano, the above shows $c_2(\mcE) < 5$, and the global generation of $\mcF = \mcE(1)$ follows from \cite[Theorem 1.7]{FHI20}, since $c_1(\mcF) = c_1(-K_X)$.

(2) follows from \cite[Corollary~4.7]{FHI20} and Proposition~\ref{prop-Ishikawa}~(6). 
\end{proof}
\begin{rem}\label{rem-leftmut}
Recall that the left mutation $\LL_{\mcO_{X}}(\mcE(1))$ is defined by the exact triangle
\[ H^{\bullet}(\mcE(1)) \otimes \mcO_X \to \mcE(1) \to \LL_{\mcO_{X}}(\mcE(1)). \]
If $\mcE(1)$ is weak Fano, then the vanishing $H^{\geq 1}(\mcE(1)) = 0$ in Proposition \ref{prop-Ishikawa} and the global generation implies $\mcH^{0}(\LL_{\mcO_{X}}(\mcE(1)))=0$.
Conversely, the vanishing $\mcH^{0}(\LL_{\mcO_{X}}(\mcE(1)))=0$ implies the global generation of $\mcE(1)$.
Thus the condition in Proposition~\ref{prop-FHI}~(1) is equivalent to saying that $\mcH^{0}(\LL_{\mcO_{X}}(\mcE(1)))=0$ and $c_{2}(\mcE) < 5$. 
\end{rem}

Here we recall a useful exact sequence and an equality given by \cite{Orlov91} or \cite[Equalities (9) and (10)]{Kuznetsov12} in studying vector bundles on del Pezzo $3$-folds of degree $5$:
\begin{align}
&0 \to \mcQ(-1) \mathop{\to}^{\coev} \Hom(\mcQ(-1),\mcR)^{\vee} \otimes \mcR \simeq \Hom(\mcR,\mcQ^{\vee}) \otimes \mcR \mathop{\to}^{\ev} \mcQ^{\vee} \to 0 \label{ex-innerECE}, \text{ and } \\
&\hom(\mcQ(-1),\mcR)=\hom(\mcR,\mcQ^{\vee})=3 \label{eq-innerECE-dim}.
\end{align}
We quickly review the above exact sequence. 
First, we note that $\hom(\mcR,\mcQ^{\vee})=3$ by \cite[Lemma~4.1]{Kuznetsov12}. 
Moreover, if we set $A:=\Hom(\mcR,\mcQ^{\vee})$, then the Hilbert scheme of lines on $X$ is isomorphic to $\P(A^{\vee})$. 
As explained in \cite[Section~4.1]{Kuznetsov12}, 
there is a non-degenerate symmetric form on $A$ and hence an isomorphism $A^{\vee} \simeq A$. 
Under the identification $\Hom(\mcQ(-1),\mcR) \simeq \Hom(\mcR^{\vee},\mcQ^{\vee}(1)) \simeq \Hom(\mcR,\mcQ^{\vee})=A$, 
the isomorphism $A^{\vee} \simeq A$ induces an isomorphism $\Hom(\mcQ(-1),\mcR)^{\vee} \simeq \Hom(\mcR,\mcQ^{\vee})$ in (\ref{ex-innerECE}). 

\subsection{Proof of Theorem~\ref{thm-resol} with $c_{2}=1$}\label{subsec-c2=1}

Suppose $c_{2}(\mcE)=1$. 
Then there exists a line $l$ on $X$ such that $\mcE$ fits in the exact sequence $0 \to \mcO_{X} \to \mcE \to \mcI_{l/X} \to 0$ by Proposition~\ref{prop-Ishikawa}~(5). 
Moreover, by \cite[Lemma~4.2]{Kuznetsov12}, there exists the following exact sequence for each line $l$:  
\begin{align}\label{ex-line-classical}
0 \to \mcR \mathop{\to}^{a_{l}} \mcQ^{\vee} \to \mcI_{l/X} \to 0. 
\end{align}
Note that $a_{l} \colon \mcR \to \mcQ^{\vee}$ factors as follows: $\displaystyle \mcR \mathop{\to}^{\id \otimes a_{l}} \mcR \otimes \Hom(\mcR,\mcQ^{\vee}) \mathop{\to}^{\ev} \mcQ^{\vee}$. 

Next we claim that using (\ref{ex-innerECE}) gives another exact sequence 
\begin{align}\label{ex-line-our}
0 \to \mcQ(-1) \to \mcR^{\oplus 2} \to \mcI_{l/X} \to 0.
\end{align}
Indeed, since $\Ext^1(\mcR, \mcQ(-1)) = 0$, the injection $\mcR \hookrightarrow \mcQ^{\vee}$ lifts to an injection $\mcR \hookrightarrow \mcR^{\oplus 3}$.
Since $\mcR$ is exceptional, the cokernel of this morphism is $\mcR^{\oplus 2}$.
This yields the following commutative diagram. 
\[ \begin{tikzcd}
&& \mcQ(-1) \arrow[d, hook] \arrow[r, equal] & \mcQ(-1) \arrow[d, hook] & \\ 
0 \arrow[r] & \mcR \arrow[d, equal] \arrow[r] & \mcR^{\oplus 3} \arrow[r] \arrow[d, two heads] & \mcR^{\oplus 2} \arrow[d, two heads] \arrow[r] & 0 \\
0 \arrow[r] & \mcR \arrow[r] & \mcQ^{\vee} \arrow[r] & \mcI_{l/X} \arrow[r] & 0 
\end{tikzcd} \]
The right vertical sequence is the desired one.

Pulling back the exact sequence (\ref{ex-line-our}) by the surjection $\mcE \to \mcI_{l/X}$, we have an extension $0 \to \mcQ(-1) \to \mcE' \to \mcE \to 0$, where $\mcE'$ is the extension of $\mcR^{\oplus 2}$ by $\mcO_{X}$. 
Since $\Ext^1(\mcR, \mcO_{X}) = 0$, the bundle $\mcE'$ is isomorphic to $\mcO_{X} \oplus \mcR^{\oplus 2}$, which gives the desired resolution 
\begin{align}\label{ex-c21-result}
0 \to \mcQ(-1) \to \mcO_{X} \oplus \mcR^{\oplus 2} \to \mcE \to 0. 
\end{align}
\begin{rem}\label{rem-c21-another}
Pulling back the exact sequence (\ref{ex-line-classical}) by the surjection $\mcE \to \mcI_{l/X}$, we also have an extension $0 \to \mcR \to \mcE'' \to \mcE \to 0$, where $\mcE''$ is the extension of $\mcQ^{\vee}$ by $\mcO_{X}$. 
Since $\Ext^{1}(\mcQ^{\vee},\mcO_{X})=H^{1}(\mcQ)=0$, 
$\mcE'' \simeq \mcQ^{\vee} \oplus \mcO_{X}$, which gives another exact sequence 
\begin{align}\label{ex-c21-another}
0 \to \mcR \to \mcQ^{\vee} \oplus \mcO_{X} \to \mcE \to 0.
\end{align}
\end{rem}

\subsection{Proof of Theorem~\ref{thm-resol} with $c_{2}=2$}
\label{subsec-c2=2}
Next, let us assume that $c_{2}(\mcE) = 2$. 
By Proposition~\ref{prop-FHI}, $\mcE$ is an instanton bundle. 
Hence \cite[Theorem~4.7]{Kuznetsov12} gives 
\begin{align}\label{ex-c22-another}
0 \to \mcR^{\oplus 2} \mathop{\to}^{\gamma'} (\mcQ^{\vee})^{\oplus 2} \to \mcE \to 0.
\end{align}
On the other hand, by (\ref{ex-innerECE}) there is an exact sequence 
$0 \to \mcQ(-1)^{\oplus 2} \to \mcR^{\oplus 6} \to (\mcQ^{\vee})^{\oplus 2} \to 0$. 
Applying $\Hom(\mcR^{\oplus 2},-)$ to this exact sequence, 
we obtain an exact sequence 
$0 \to \Hom(\mcR^{\oplus 2},\mcQ(-1)^{\oplus 2}) \to \Hom(\mcR^{\oplus 2},\mcR^{\oplus 6}) \to \Hom(\mcR^{\oplus 2},(\mcQ^{\vee})^{\oplus 2}) \to \Ext^{1}(\mcR^{\oplus 2},\mcQ(-1)^{\oplus 2})$. 
Since $\Ext^{\bullet}(\mcR,\mcQ(-1)) = 0$, 
the map 
$\Hom(\mcR^{\oplus 2},\mcR^{\oplus 6}) \to \Hom(\mcR^{\oplus 2},(\mcQ^{\vee})^{\oplus 2})$ 
is an isomorphism. 
Hence $\gamma'$ uniquely lifts to a morphism $\wt{\gamma} \colon \mcR^{\oplus 2} \hra \mcR^{\oplus 6}$. 
Since $\gamma'$ is injective, so is $\wt{\gamma}$. 
Note that $\mcR$ is simple and hence the natural map $\Hom(\Bbbk^{\oplus 2},\Bbbk^{\oplus 6}) \simeq \Hom(\mcR^{\oplus 2},\mcR^{\oplus 6})$ is an isomorphism. 
Thus the cokernel of $\wt{\gamma}$ is isomorphic to $\mcR^{\oplus 4}$: 
\[\xymatrix{
0 \ar[r] & \mcR^{\oplus 2} \ar[r]^{\gamma'} & (\mcQ^{\vee})^{\oplus 2} \ar[r] & \mcE \ar[r] & 0 \\ 
0 \ar[r]& \mcR^{\oplus 2} \ar[r]^{\wt{\gamma}} \ar@{=}[u]& \mcR^{\oplus 6} \ar[u] \ar[r]& \mcR^{\oplus 4} \ar@{..>}[u] \ar[r]& 0
}.\]
Thus we obtain an exact sequence 
\begin{align}\label{ex-c22-our}
0 \to \mcQ(-1)^{\oplus 2} \to \mcR^{\oplus 4} \to \mcE \to 0.
\end{align}

\subsection{Preliminaries for $c_{2} \geq 3$}

When $c_{2}(\mcE) \geq 3$, we will use another technique whose key ingredient is Theorem~\ref{Orlov}. 
Let us prepare the following lemma for instanton bundles as a tool for finding the desired resolutions.

\begin{lem}\label{lem-tool-instanton}
For an instanton bundle $\mcE$ on $X$, 
there is a distinguished triangle
\begin{align}\label{dt-instanton}
&H^{1}(\mcE) \otimes \mcO_{X}(-1) \to \RR_{\mcO_{X}(-1)}\LL_{\mcO_{X}}(\mcE(1))[-1] \to \LL_{\mcO_{X}}(\mcE(1))[-1] \mathop{\to}^{+1} 
\end{align}
Moreover, if $\mcE$ is weak Fano, then 
there is a vector bundle $\mcV$ such that 
$\mcV \simeq \RR_{\mcO_{X}(-1)}\LL_{\mcO_{X}}(\mcE(1))[-1]$ fitting into the following exact sequence; 
\begin{align}\label{ex-instanton-wF}
0 \to H^{1}(\mcE) \otimes \mcO_{X}(-1) \to \mcV \to H^{0}(\mcE(1)) \otimes \mcO_{X} \mathop{\to}^{\ev} \mcE(1) \to 0. 
\end{align}
\end{lem}

\begin{proof}
Consider the right mutation of $\LL_{\mcO_{X}}(\mcE(1))$ by $\mcO_{X}(-1)$:
\[\RR_{\mcO_{X}(-1)}\LL_{\mcO_{X}}(\mcE(1)) \to \LL_{\mcO_{X}}(\mcE(1)) \to \RHom(\LL_{\mcO_{X}}(\mcE(1)),\mcO_{X}(-1))^{\vee} \otimes \mcO_{X}(-1) \mathop{\to}^{+1}\]
Using the distinguished triangle 
$\displaystyle \RHom(\mcO_{X},\mcE(1)) \otimes \mcO_{X} \to \mcE(1) \to \LL_{\mcO_{X}}(\mcE(1)) \mathop{\to}^{+1} $ and the Serre duality, 
we have 
\begin{align*}
\RHom(\LL_{\mcO_{X}}(\mcE(1)),\mcO_{X}(-1))^{\vee} \simeq &\RHom(\mcE(1),\mcO_{X}(-1))^{\vee} \\
\simeq &\RHom(\mcE,\mcO_{X}(K_{X}))^{\vee} \\
\simeq &\RHom(\mcO_{X},\mcE)[-3] \simeq H^{1}(\mcE)[2],
\end{align*}
where the last isomorphism follows from \cite[Lemma~3.1]{Kuznetsov12}. 
This gives the triangle (\ref{dt-instanton}). 

Let us suppose that $\mcE$ is weak Fano. 
By Proposition~\ref{prop-FHI}, $\mcF:=\mcE(1)$ is globally generated. 
Thus $\LL_{\mcO_{X}}(\mcE(1))[-1] \simeq \Ker(H^{0}(\mcE(1)) \otimes \mcO_{X} \epm \mcE(1))$ is also a locally free sheaf. 
Then by (\ref{dt-instanton}), 
$\mcV:=\RR_{\mcO_{X}(-1)}\LL_{\mcO_{X}}(\mcE(1))[-1]$ is a locally free sheaf. 
Thus we obtain the exact sequence (\ref{ex-instanton-wF}). 
This completes the proof. 
\end{proof}
%
%
In particular, for a given rank $2$ weak Fano bundle $\mcE$ with $c_{1}(\mcE)=0$ and $c_{2}(\mcE) \geq 2$, 
a resolution of $\mcE(1)$ can be obtained by computing $\mcV=\mcH^{-1}(\RR_{\mcO_{X}(-1)}\LL_{\mcO_{X}}(\mcE(1)))$ in the exact sequence (\ref{ex-instanton-wF}). 
For computing $\mcV$, we need the following lemma. 
\begin{lem}\label{lem-mutation-dP5}
Let $\mcE \in \Db(X)$ be an object. 
Then we have the following isomorphisms.
\begin{enumerate}
\item $\RHom(\RR_{\mcO_{X}(-1)}\LL_{\mcO_{X}}(\mcE(1)),\mcQ(-1)) \simeq \RHom(\mcR,\mcE)^{\vee}[-2]$. 
\item $\RHom(\mcR,\RR_{\mcO_{X}(-1)}\LL_{\mcO_{X}}(\mcE(1))) \simeq \RHom(\mcQ(-1),\mcE)[1]$. 
\end{enumerate}
In particular, applying Lemma~\ref{lem-2term} to the object $\RR_{\mcO_{X}(-1)}\LL_{\mcO_{X}}(\mcE(1)) \in {}^{\bot}\braket{\mcO_{X}(-1)} \cap \braket{\mcO_{X}}^{\bot}=\braket{\mcQ(-1),\mcR}$, 
we have the following distinguished triangle
\begin{align}\label{dt-Serre}
\RHom(\mcQ(-1),\mcE)[1] \otimes \mcR \to \RR_{\mcO_{X}(-1)}\LL_{\mcO_{X}}(\mcE(1)) \to \RHom(\mcR,\mcE)[2] \otimes \mcQ(-1) \mathop{\to}^{+1}.
\end{align}
\end{lem}

\begin{proof}
First, we show the following claim. 

\begin{claim}\label{claim-Serre-dP5}
Set $\mcA=\braket{\mcQ(-1),\mcR}$ and 
$\mcB:=\braket{\mcO_{X}(-1),\mcQ(-1),\mcR}$. 
Note that $\mcA \subset \mcB$ and $\mcB \subset \Db(X)$ are admissible subcategories. 
Let $S_{\mcA}$ and $S_{\mcB}$ be the Serre functors of $\mcA$ and $\mcB$ induced by the Serre functor of $\Db(X)$ respectively. 
Then we have the following assertions.
\begin{enumerate}
\item $S_{\mcA}^{-1}(\mcQ(-1)) \simeq \mcQ^{\vee}[-1]$. 
\item $S_{\mcB}(\mcQ^{\vee}) \simeq \mcR(-1)[2]$. 
\item $S_{\mcB}(\mcR) \simeq \mcQ(-2)[2]$. 
\end{enumerate}
\end{claim}

\begin{proof}[Proof of Claim~\ref{claim-Serre-dP5}]

Since there is a semi-orthogonal decomposition $\Db(X) = \langle \mcA, \mcO_{X}, \mcO_{X}(1) \rangle$, the inverse Serre functor $S_{\mcA}^{-1}$ can be computed as 
$S_{\mcA}^{-1} \simeq \LL_{\mcO_{X}}\LL_{\mcO_{X}(1)}(- \otimes \mcO_{X}(2)[-3])$.
Thus $S_{\mcA}^{-1}(\mcQ(-1)) \simeq \LL_{\mcO_{X}}\LL_{\mcO_{X}(1)}(\mcQ(1))[-3]$, and exact sequences 
$0 \to \mcR(1) \to \mcO_{X}(1)^{\oplus 5} \to \mcQ(1) \to 0$
and
$0 \to \mcQ^{\vee} \to \mcO_{X}^{\oplus 5} \to \mcR(1) \to 0 $
give 
\[ \LL_{\mcO_{X}}\LL_{\mcO_{X}(1)}(\mcQ(1))[-3] \simeq  \LL_{\mcO_{X}}(\mcR(1)[1])[-3] \simeq \mcQ^{\vee}[-1]. \] 
This shows (1).

Similarly, to prove (2) and (3), let us consider a semi-orthogonal decomposition $\Db(X) = \langle \mcO_{X}(-2), \mcB \rangle$, 
which gives the functor isomorphism $S_{\mcB} \simeq \RR_{\mcO_{X}(-2)}(- \otimes \mcO_{X}(-2)[3])$.
Then the exact sequences 
$0 \to \mcQ^{\vee}(-2) \to \mcO_{X}(-2)^{\oplus 5} \to \mcR(-1) \to 0$ 
and 
$0 \to \mcR(-2) \to \mcO_{X}(-2)^{\oplus 5} \to \mcQ(-2) \to 0$ 
give computations 
$S_{\mcB}(\mcQ^{\vee}) \simeq \RR_{\mcO_{X}(-2)}(\mcQ^{\vee}(-2))[3] \simeq \mcR(-1)[2]$ 
and 
$S_{\mcB}(\mcR) \simeq \RR_{\mcO_{X}(-2)}(\mcR(-2))[3] \simeq  \mcQ(-2)[2]$ 
respectively. 
We complete the proof of Claim~\ref{claim-Serre-dP5}. 
\end{proof}
Using Claim~\ref{claim-Serre-dP5}, we prove Lemma~\ref{lem-mutation-dP5}. 

(1) Since $\RR_{\mcO_{X}(-1)}\LL_{\mcO_{X}}(\mcE(1)),\mcQ(-1) \in \mcA$, 
we have 
\begin{align*}
&\RHom_{\Db(X)}(\RR_{\mcO_{X}(-1)}\LL_{\mcO_{X}}(\mcE(1)),\mcQ(-1))  \\
=&\RHom_{\mcA}(\RR_{\mcO_{X}(-1)}\LL_{\mcO_{X}}(\mcE(1)),\mcQ(-1)) \\
=&\RHom_{\mcA}(S_{\mcA}^{-1}(\mcQ(-1)),\RR_{\mcO_{X}(-1)}\LL_{\mcO_{X}}(\mcE(1)))^{\vee} \quad \text{(Serre duality in $\mcA$)} \\
=&\RHom_{\mcB}(S_{\mcA}^{-1}(\mcQ(-1)),\LL_{\mcO_{X}}(\mcE(1)))^{\vee} \quad 
(\text{Remark~\ref{rem-cat}~(1)}) \\
=&\RHom_{\mcB}(\LL_{\mcO_{X}}(\mcE(1)),S_{\mcB}S_{\mcA}^{-1}(\mcQ(-1))) \quad \text{(Serre duality in $\mcB$)} \\
=&\RHom_{\Db(X)}(\mcE(1),S_{\mcB}S_{\mcA}^{-1}(\mcQ(-1)))\quad 
(\text{Remark~\ref{rem-cat}~(1)}) \\
=&\RHom_{\Db(X)}(\mcE(1),\mcR(-1)[1]) \quad (\text{Claim~\ref{claim-Serre-dP5}}) \\
=&\RHom_{\Db(X)}(\mcR,\mcE)^{\vee}[-2] \quad (\text{Serre duality on $X$}). 
\end{align*}

(2) Since $\mcR,\RR_{\mcO_{X}(-1)}\LL_{\mcO_{X}}(\mcE(1)) \in \mcA$, 
we have 
\begin{align*}
&\RHom_{\Db(X)}(\mcR,\RR_{\mcO_{X}(-1)}\LL_{\mcO_{X}}(\mcE(1))) \\
=&\RHom_{\mcA}(\mcR,\RR_{\mcO_{X}(-1)}\LL_{\mcO_{X}}(\mcE(1))) \\
=&\RHom_{\mcB}(\mcR,\LL_{\mcO_{X}}(\mcE(1))) \quad 
(\text{Remark~\ref{rem-cat}~(1)}) \\
=&\RHom_{\mcB}(\LL_{\mcO_{X}}(\mcE(1)),S_{\mcB}(\mcR))^{\vee} \quad \text{(Serre duality in $\mcB$)} \\
=&\RHom_{\Db(X)}(\mcE(1),S_{\mcB}(\mcR))^{\vee} \quad 
(\text{Remark~\ref{rem-cat}~(1)}) \\
=&\RHom_{\Db(X)}(\mcE(1),\mcQ(-2)[2])^{\vee} \quad (\text{Claim~\ref{claim-Serre-dP5}}) \\
=&\RHom_{\Db(X)}(\mcQ(-1),\mcE)[1] \quad (\text{Serre duality on $X$}).
\end{align*}
This completes the proof of Lemma~\ref{lem-mutation-dP5}. 
\end{proof}
In particular, the computation of $\mcV$ can be reduced to computing the cohomologies of $\RHom(\mcQ(-1),\mcE)$ and $\RHom(\mcR,\mcE)$. 
We also need the following lemma to compute them. 
\begin{lem}\label{lem-extgeq2}
It holds that $\Ext^{\geq 2}(\mcQ(-1),\mcE)=\Ext^{\geq 2}(\mcR,\mcE)=0$ for an instanton bundle $\mcE$. 
\end{lem}
\begin{proof}
First of all, $H^{\geq 2}(\mcE(1))=0$ follows from \cite[Lemma~3.1]{Kuznetsov12}. 
Hence $\mcH^{\geq 1}(\RR_{\mcO_{X}(-1)}\LL_{\mcO_{X}}(\mcE(1)))=0$ by (\ref{dt-instanton}). 
Then by (\ref{dt-Serre}), it suffices to show $\Ext^{i}(\mcR,\mcE) \simeq H^{i}(\mcR^{\vee} \otimes \mcE)=0$ for $i \geq 2$. 

Let $s \colon \mcO_{X} \to \mcR^{\vee}$ be a general section and consider the following exact sequence:
\[0 \to \mcE \mathop{\to}^{s \otimes \id_{\mcE}} \mcR^{\vee} \otimes \mcE \to \mcI_{C/X}(1) \otimes \mcE \to 0,\]
where $C$ is the conic corresponding to $s$. 
It is well-known that every conic $C$ on $X$ can be realized as the zero scheme of a global section of $\mcR^{\vee}$ (see \cite[Theorem~5.6]{CZ19} and its reference). 
Since $H^{\geq 2}(\mcE)=0$ by \cite[Lemma~3.1]{Kuznetsov12},  
it is enough to show that $H^{\geq 2}(\mcI_{C/X}(1) \otimes \mcE)=0$ for a general conic $C$. 
This vanishing holds since it was shown by \cite[Theorem~4.24]{Sanna14} that $\mcE|_{C} \simeq \mcO_{\P^{1}}^{\oplus 2}$ for a general conic $C$. 
We complete the proof.
\end{proof}

The conclusion of this section is as follows. 
\begin{lem}\label{lem-conclu}
Let $\mcE$ be an instanton bundle on $X$. 
\begin{enumerate}
\item There is an exact sequence
\begin{align}\label{ex-tau}
\Ext^{1}(\mcR,\mcE) \otimes \mcQ(-1) \mathop{\to}^{\tau} \Ext^{1}(\mcQ(-1),\mcE) \otimes \mcR \to \mcH^{0}(\LL_{\mcO_{X}}(\mcE(1))) \to 0.
\end{align}
\item $\mcE$ is weak Fano if and only if $\tau$ is surjective and $c_{2}(\mcE) < 5$. 
\item If $\mcE$ is weak Fano, there is an exact sequence
\begin{align}\label{ex-2terms}
0 \to \Hom(\mcR,\mcE) \otimes \mcQ(-1) \to \Hom(\mcQ(-1),\mcE) \otimes \mcR \to \mcV \to \Ext^{1}(\mcR,\mcE) \otimes \mcQ(-1)  \mathop{\to}^{\tau} \Ext^{1}(\mcQ(-1),\mcE) \otimes \mcR \to 0,
\end{align}
where $\mcV$ is the vector bundle defined in Lemma~\ref{lem-tool-instanton}.
\end{enumerate}
\end{lem}
\begin{proof}
(1) Since $\LL_{\mcO_{X}}(\mcE(1))$ is concentrated in $\{0,-1\}$, 
so is $\RR_{\mcO_{X}(-1)}\LL_{\mcO_{X}}(\mcE(1))$. 
Thus the morphism $\mcH^{0}(\RR_{\mcO_{X}(-1)}\LL_{\mcO_{X}}(\mcE(1))) \to \mcH^{0}(\LL_{\mcO_{X}}(\mcE(1)))$ induced by (\ref{dt-instanton}) is an isomorphism. 
Then by taking cohomology of the distinguished triangle (\ref{dt-Serre}) and using Lemma~\ref{lem-extgeq2}, 
we have (\ref{ex-tau}). 

(2) follows from Proposition~\ref{prop-FHI} and Remark~\ref{rem-leftmut}. 

(3) If $\mcE$ is weak Fano, then $\mcH^{0}(\RR_{\mcO_{X}(-1)}\LL_{\mcO_{X}}(\mcE(1)))=\mcH^{0}(\LL_{\mcO_{X}}(\mcE(1)))=0$ as in Remark~\ref{rem-leftmut}. Then (\ref{ex-2terms}) follows from (\ref{dt-instanton}), (\ref{dt-Serre}), and Lemma~\ref{lem-extgeq2}. 
\end{proof}

\subsection{Proof of Theorem~\ref{thm-resol} with $c_{2}=3$.}\label{subsec-c2=3}

Let $\mcE$ be a rank $2$ weak Fano bundle with $c_{1}(\mcE)=0$ and $c_{2}(\mcE)=3$. 
Then $\mcE(1)$ is globally generated and $\mcE$ is an instanton bundle by Proposition~\ref{prop-FHI}. 
Hence $\mcE|_{l} \simeq \mcO_{X}^{2}$ or $\mcO_{X}(-1) \oplus \mcO_{X}(1)$ for every line $l$ on $X$. 
Then it follows from \cite[Definition~8.2 and Proposition~8.7]{Sanna17} that $\RHom(\mcR,\mcE) \simeq \Bbbk$. 
Then (\ref{ex-2terms}) gives an exact sequence 
\[0 \to \mcQ(-1) \mathop{\to}^{\alpha} \mcR^{\oplus a} \to \mcV \to 0, \]
where $\mcV:=\mcH^{-1}(\RR_{\mcO_{X}(-1)}\LL_{\mcO_{X}}(\mcE(1)))$ and $a=\hom(\mcQ(-1),\mcE)$. 
Note that $\RG(\mcE(1))=\Bbbk^{\oplus 8}$ and $\RG(\mcE)=\Bbbk[-1]$ by Proposition~\ref{prop-Ishikawa}~(1), (2) and (6). 
Then the exact sequence (\ref{ex-instanton-wF}) in Lemma~\ref{lem-tool-instanton} shows $\rk \mcV=7$, which implies $a=5$. 
Set $\mcK:=\Ker(\ev \colon H^{0}(\mcE(1)) \otimes \mcO_{X} \to \mcE(1))$ and $\mcK':=\Ker(\mcR^{\oplus 5} \epm \mcV \epm \mcK)$. 
Since $h^{1}(\mcE)=1$, 
we have the following diagram by (\ref{ex-instanton-wF}): 
\[\xymatrix{
&\mcQ(-1) \ar@{^{(}->}[d] \ar@{=}[r]& \mcQ(-1) \ar@{^{(}->}[d]&& \\
0 \ar[r] & \mcK' \ar[r] \ar@{->>}[d]& \mcR^{\oplus 5} \ar[r] \ar@{->>}[d]& \mcK \ar[r] \ar@{=}[d]& 0 \\
0 \ar[r] & \mcO_{X}(-1) \ar[r] & \mcV \ar[r] & \mcK \ar[r] & 0.
}\]
Since $\Ext^{1}(\mcO_{X}(-1),\mcQ(-1))=0$, 
we have $\mcK'=\mcO_{X}(-1) \oplus \mcQ(-1)$. 
Regarding this exact sequence $0 \to \mcO_{X}(-1) \oplus \mcQ(-1) \to \mcR^{\oplus 5} \to \mcK \to 0$ as a resolution of $\mcK$, 
we have the desired resolution 
\[0 \to \mcO_{X}(-1) \oplus \mcQ(-1) \to \mcR^{\oplus 5} \to \mcO_{X}^{\oplus 8} \to \mcE(1) \to 0.\qed\]

\begin{rem}\label{rem-c23-another}
We can give another resolution by showing $\mcV \simeq \mcR^{\oplus 2} \oplus \mcQ^{\vee}$. 
By the universality of the co-evaluation map (see (\ref{ex-innerECE})), 
the map $\alpha \colon \mcQ(-1) \to \mcR^{\oplus 5}$ factors as follows: 
\[\xymatrix{
0 \ar[r] & \mcQ(-1) \ar[r]^{\alpha} & \mcR^{\oplus 5} \ar[r] & \mcV \ar[r] & 0 \\
0 \ar[r] & \mcQ(-1) \ar[r]^{\coev} \ar@{=}[u]&\mcR^{\oplus 3} \ar[r] \ar[u]_{\ol{\a}}& \mcQ^{\vee} \ar[r] \ar[u]& 0 \\
}\]
We claim $\ol{\a}$ is injective. 
To obtain a contradiction, suppose that $\ol{\a}$ is not injective. 
Since $\Im \ol{\a}$ contains $\Im \a$, $\rk \Im \ol{\a} \geq 3$. 
Since $\mcR$ is simple and $\Hom(\mcR^{\oplus 3},\mcR^{\oplus 5}) \simeq \Hom(\Bbbk^{\oplus 3},\Bbbk^{\oplus 5})$, 
it holds that $\Im(\ol{\a}) \simeq \mcR^{\oplus 2} \subset \mcR^{\oplus 5}$. 
Then it follows from (\ref{ex-line-our}) that $\mcV \simeq \mcI_{l/X} \oplus \mcR^{\oplus 3}$ for some line $l$, which contradicts that $\mcV$ is locally free. 
Hence $\ol{\a}$ is injective. 

Therefore, we obtain $0 \to \mcQ^{\vee} \to \mcV \to \mcR^{\oplus 2} \to 0$. 
Since $\Ext^{1}(\mcR,\mcQ^{\vee})=0$ (c.f. \cite{Orlov91}, \cite[Lemma~4.1]{Kuznetsov12}), 
we have $\mcV=\mcR^{\oplus 2} \oplus \mcQ^{\vee}$. 
By Lemma~\ref{lem-tool-instanton}, we obtain another resolution
\[0 \to \mcO_{X}(-1) \to \mcR^{\oplus 2} \oplus \mcQ^{\vee} \to \mcO_{X}^{\oplus 8} \to \mcE(1) \to 0.\]
\end{rem}

\subsection{Proof of Theorem~\ref{thm-resol} with $c_{2}=4$.}\label{subsec-c2=4}

Let $\mcE$ be a rank $2$ weak Fano bundle with $c_{1}(\mcE)=0$ and $c_{2}(\mcE)=4$. 
In this case, we need the following inequality, which will be used to indicate that there is no morphism between certain vector bundles. 

\begin{lem}\label{lem-codim2class}
Let $X$ be a smooth projective $3$-fold. 
Let $\mcF$ be a nef locally free sheaf and $\mcG \subset \mcF$ a subsheaf of rank $2$. 
Let $\pi \colon \P_{X}(\mcF) \to X$ be the projectivization of $\mcF$ and $\xi$ a tautological divisor. 

Suppose that $\Hom(\mcG(B),\mcF)=0$ for every non-zero effective divisor $B>0$ on $X$. 
Then the following inequality holds: 
\begin{align}\label{ineq-codim2}
0\leq &(c_{1}(\mcF)^{3}-2c_{1}(\mcF)c_{2}(\mcF)+c_{3}(\mcF))-(c_{1}(\mcF)^{2}-c_{2}(\mcF))\ch_{1}(\mcG)\\
&+\frac{1}{2}c_{1}(\mcF)\left(\ch_{1}(\mcG)^{2}-2\ch_{2}(\mcG)\right). \notag
\end{align}
\end{lem}
\begin{proof}
Let $\pi^{\ast}\mcG \to \mcO_{\P(\mcF)}(\xi)$ be the morphism corresponding to the inclusion $i \colon \mcG \hra \mcF$. 
Tensoring with $\mcO_{\P(\mcF)}(-\xi)$, we obtain a map 
$\pi^{\ast}\mcG(-\xi) \to \mcO_{\P(\mcF)}$. 
Let $\mcI$ be its image, which is an ideal sheaf. 
Then there is an effective divisor $D$ and a closed subscheme $Z$ with $\codim_{\P(\mcF)}Z \geq 2$ such that $\mcI=\mcI_{Z/\P(\mcF)}(-D)$. 
Let $\alpha \colon \pi^{\ast}\mcG \epm \mcI_{Z/\P(\mcF)}(-D) \otimes \mcO_{\P(\mcF)}(\xi)$ be the surjection obtained by this way. 

Let us show $D=0$. 
Since $D$ is an effective divisor, there are $a \geq 0$ and a divisor $B$ on $X$ such that $D \sim a\xi+\pi^{\ast}B$. 
Since the push-forward of the composition $\pi^{\ast}\mcG \epm \mcI_{Z/\P(\mcF)}((1-a)\xi-\pi^{\ast}B) \to \mcO_{\P(\mcF)}(1)$ is the original inclusion $\mcG \hra \mcF$, 
this inclusion factors $\Sym^{1-a}\mcF \otimes \mcO(-B)$. 
Hence $a=0$ and $D \sim \pi^{\ast}B$, which implies $B$ is effective since so is $D$. 
Then $B=0$ follows from our assumption. 
Hence $D=0$. 

Now we obtain a surjection $\a \colon \pi^{\ast}\mcG \epm \mcI_{Z/\P(\mcF)}(\xi)$. 
Since the projection $\pi$ is a flat morphism, $\pi^{\ast}\mcG$ remains to be torsion free, and hence so is $\Ker \a$. 
Since we suppose that $\rk \mcG=2$, there is an exact sequence 
\[0 \to \pi^{\ast}\det \mcG \otimes \mcO_{\P(\mcF)}(-\xi) \otimes \mcI_{W/\P(\mcF)} \to \pi^{\ast}\mcG \to \mcI_{Z/\P(\mcF)}(\xi) \to 0,\]
where $W \subset X$ is a closed subscheme of $\codim_{X}W \geq 2$ and $\det \mcG$ is the determinant invertible sheaf of the torsion free sheaf $\mcG$ defined as in \cite[Definition~in~P.129]{SRS}.  
Hence we obtain that 
$\ch(\pi^{\ast}\mcG)=\ch(\pi^{\ast}\det \mcG \otimes \mcO_{\P(\mcE)}(-\xi) \otimes \mcI_{W/\P(\mcF)})+\ch(\mcI_{Z/\P(\mcF)}(\xi))$. 
Note that 
$\ch(\mcI_{Z/\P(\mcF)})=1-\ch(\mcO_{Z})$, 
$\ch(\mcI_{W/\P(\mcF)})=1-\ch(\mcO_{W})$, and 
$\ch_{k}(\mcO_{Z})=\ch_{k}(\mcO_{W})=0$ for $k \leq 1$ since $\codim_{X}Z \geq 2$ and $\codim_{X}W \geq 2$. 
Thus we obtain the following equality by direct computation:
\begin{align}\label{eq-codim2eff}
\ch_{2}(\mcO_{Z})+\ch_{2}(\mcO_{W})
&=\xi^{2}-\xi.\pi^{\ast}\ch_{1}(\mcG)+\frac{1}{2}\left(\ch_{1}(\mcG)^{2}-2\ch_{2}(\mcG)\right).
\end{align}
%
%
Therefore the right hand side of the above equality $\xi^{2}-\xi.\pi^{\ast}\ch_{1}(\mcG)+\frac{1}{2}\left(\ch_{1}(\mcG)^{2}-2\ch_{2}(\mcG)\right)$ is numerically equivalent to an effective codimension $2$ cycle. 
Since $\mcF$ is nef, 
letting $n:=\dim \P_{X}(\mcF)=\dim X+\rk \mcF-1$, 
we have 
\begin{align} \label{ineq-codim2-pre}
0 \leq \xi^{n}-\xi^{n-1}.\pi^{\ast}\ch_{1}(\mcG)+\xi^{n-2}.\frac{1}{2}\left(\ch_{1}(\mcG)-2\ch_{2}(\mcG)\right) 
\end{align}
by (\ref{eq-codim2eff}). 
Since $\pi_{\ast}\xi^{n}=c_{1}(\mcF)^{3}-2c_{1}(\mcF)c_{2}(\mcF)+c_{3}(\mcF)$, the equalities $\pi_{\ast}\xi^{n-1}=c_{1}(\mcF)^{2}-c_{2}(\mcF)$ and $\pi_{\ast}\xi^{n-2}=c_{1}(\mcF)$ hold,
and evaluating them to (\ref{ineq-codim2-pre}) gives the desired inequality (\ref{ineq-codim2}). 
\end{proof}

Now Lemma~\ref{lem-codim2class} enables us to compute the following derived Hom spaces. 

\begin{lem}\label{lem-vani-c24}
$\RHom(\mcR,\mcE) \simeq \Bbbk^{2}[-1]$ and $\RHom(\mcQ(-1),\mcE) \simeq 0$. 
\end{lem}

\begin{proof}
The Hirzebruch-Riemann-Roch theorem gives 
$\chi(\mcQ(-1),\mcE)=0$
and 
$\chi(\mcR,\mcE)=-2$.
Since $\Ext^{\geq 2}(\mcQ(-1),\mcE)=\Ext^{\geq 2}(\mcR,\mcE)=0$ by Lemma~\ref{lem-extgeq2}, 
it is enough to show $\Hom(\mcR,\mcE)=\Hom(\mcQ(-1),\mcE)=0$. 
By (\ref{ex-2terms}), it suffices to show only $\Hom(\mcQ(-1),\mcE)=0$. 
Set $\mcF:=\mcE(1)$. Then the above equality is equivalent to 
$\Hom(\mcQ,\mcF)=0$. 

Assume the contrary and let $s \colon \mcQ \to \mcF$ be a non-zero morphism. 
Set $\mcG:=\Im(s)$, $\mcK:=\Ker(s)$, and $\mcT:=\mcF/\mcG$. 
Then $\mcK$ is reflexive and $\mcG$ is torsion free. 
Let $\alpha \colon \mcG \hra \mcF$ be the inclusion map. 
Then we have the following two exact sequences 
\begin{align}
&0 \to \mcK \to \mcQ \to \mcG \to 0, \text{ and } \label{ex-c241st}\\
&0 \to \mcG \mathop{\to}^{\alpha} \mcF \to \mcT \to 0. \label{ex-c242nd}
\end{align}

Suppose that $\rk \mcG=1$. 
Let $a \in \Z$ be an integer such that $\mcG^{\vee\vee} \simeq \mcO_{X}(a)$. 
Then there are non-zero morphisms $\mcQ \to \mcO_{X}(a)$ and $\mcO_{X}(a) \to \mcF$. 
Since $\Hom(\mcQ,\mcO_{X})=0$, we have $a \geq 1$. 
Conversely, since $H^{0}(\mcE)=0$, we have $a \leq 0$. 
This is a contradiction. 

Hence $\mcG$ is torsion free and of rank $2$.
Then $\mcK$ is a reflexible sheaf of rank $1$, which means there is an integer $a$ such that $\mcK=\mcO_{X}(-a)$. 
Note that the inclusion $\mcO_{X}(-a) \simeq \mcK \subset \mcQ$ implies $a \geq 0$.
In addition, since $\mcT$ is a torsion sheaf, $c_1(\mcT) \geq 0$, and hence the sequence (\ref{ex-c242nd}) together with the formulas
$c_{1}(\mcG)=c_{1}(\mcG^{\vee\vee})=1+a$ and $c_{1}(\mcF)=2$ implies $a \leq 1$. 
Hence $a \in \{0,1\}$.

The slope stability shows that there is no non-zero effective divisor $B$ such that $\Hom(\mcQ(B),\mcF) \neq 0$. 
Thus the surjection $\mcQ \twoheadrightarrow \mcG$ yields that there is no non-zero effective divisor $B$ such that $\Hom(\mcG(B),\mcF) \neq 0$. 
Thus the inequality in Lemma~\ref{lem-codim2class} gives $10a^{2}-(a+1) \geq 0$,
which implies $a=1$ and hence $\mcK \simeq \mcO_{X}(-1)$. 

Now (\ref{ex-c241st}) gives $\RG(\mcG) \simeq \RG(\mcQ) \simeq \Bbbk^{\oplus 5}$. 
Since $\RG(\mcF)=\Bbbk^{\oplus 6}$ by Proposition~\ref{prop-Ishikawa}~(1) and (2), 
it holds that $\RG(\mcT) \simeq \Bbbk$ by (\ref{ex-c242nd}). 
Since $\mcF$ is globally generated, so is $\mcT$. 
Let $t \colon \mcO_{X} \to \mcF$ be a section such that 
the composition $\displaystyle \mcO_{X} \mathop{\to}^{t} \mcF \to \mcT$ is surjective. 
Then we obtain a surjection $\mcG \epm \Cok t$. 
Composing this with $\mcQ \epm \mcG$ gives a surjection $\beta \colon \mcQ \to \Cok t$. 
Since $\Cok t$ has a resolution $0 \to \mcO_{X} \to \mcF \to \Cok t \to 0$, 
the argument using the depth of sheaves shows that
the sheaf $\mcH:=\Ker \beta$ that fits in an exact sequence
\[ 0 \to \mcH \to \mcQ \xrightarrow{\beta} \Cok t \to 0 \]
 is a locally free sheaf of rank $2$.
%
%
%
Since $c_{1}(\mcF)=2H_{X}$, $c_{2}(\mcF)=9l$ and $c_{3}(\mcF)=0$, 
a computation using the formulas
$\ch_{1}(\mcF)=c_{1}(\mcF)$, 
$\ch_{2}(\mcF)=(1/2)(c_{1}(\mcF)^{2}-2c_{2}(\mcF))$, and
$\ch_{3}(\mcF)=(1/6)(c_{1}(\mcF)^{3}-3c_{1}(\mcF)c_{2}(\mcF)+3c_{3}(\mcF))$
yields
\[ \ch(\mcF)=\left(2,2H_{X},l,-\frac{7}{3} \right). \]
A similar calculation gives $\ch(\mcQ)=\left(3,H_{X},-\frac{1}{2}l,-\frac{1}{6} \right)$. 
Hence we have 
\[ \ch(\mcH)=\ch(\mcQ)-\ch(\Cok t)=\ch(\mcQ)-(\ch(\mcF)-\ch(\mcO_{X}))=\left(2,-H_{X},-\frac{3}{2}l,\frac{13}{6} \right), \]
and thus
\[ c_{3}(\mcH)=\frac{1}{6}\ch_{1}(\mcH)^{3}-\ch_{1}(\mcH)\ch_{2}(\mcH)+2\ch_{3}(\mcH)=2. \] 
However, since $\mcH$ is locally free of rank $2$, $c_{3}(\mcH)=0$. 
This is a contradiction. 
Hence $\Hom(\mcQ,\mcF)=0$, which completes the proof of Lemma~\ref{lem-vani-c24}. 
\end{proof}
Then (\ref{ex-2terms}) shows $\mcV=\mcQ(-1)^{\oplus 2}$ and the desired resolution
\[0 \to \mcO_{X}(-1)^{\oplus 2} \to \mcQ(-1)^{\oplus 2} \to \mcO_{X}^{\oplus 6} \to \mcE(1) \to 0. \]
We complete the proof of Theorem~\ref{mainthm-resol} when $c_{2}(\mcE)=4$. \qed

As a consequence of results in this section, we obtain the following characterization, which is stronger than Lemma~\ref{lem-conclu}. 
\begin{prop}\label{prop-chara}
Let $\mcE$ be a rank $2$ vector bundle on $X$ with $c_{1}(\mcE)=0$. 
Then $\mcE$ is a stable weak Fano bundle if and only if $\mcE$ is an instanton bundle with $\Ext^{1}(\mcQ(-1),\mcE)=0$. 
\end{prop}
\begin{proof}
The implication $\la$ is clear from (\ref{ex-tau}).
Suppose that $\mcE$ is a stable weak Fano bundle with $c_{1}(\mcE)=0$ and $c_{2}(\mcE) \geq 2$. 
Then $\mcE$ is an instanton bundle and $\mcE(1)$ is globally generated 
by Proposition~\ref{prop-FHI}. 
By Proposition~\ref{prop-Ishikawa}~(3) and (6), we have $c_{2} \in \{2,3,4\}$. 
If $c_{2}(\mcE)=2$, then $\mcE$ satisfies (\ref{ex-c21-result}), which implies $\Ext^{1}(\mcQ(-1),\mcE)=0$. 
If $c_{2}(\mcE)=3$, then $\RHom(\mcR,\mcE)=\Bbbk$ as explained in the beginning of Section~\ref{subsec-c2=3}. 
Then $\Ext^{1}(\mcQ(-1),\mcE)=0$ follows from (\ref{ex-2terms}). 
If $c_{2}(\mcE)=4$, then $\Ext^{1}(\mcQ(-1),\mcE)=0$ by Lemma~\ref{lem-vani-c24}. 
Hence we have $\Ext^{1}(\mcQ(-1),\mcE)=0$ for each case. 
\end{proof}

\subsection{Proof of Theorem~\ref{mainthm-resol}}
Now we classify rank $2$ weak Fano bundles on a del Pezzo $3$-fold of degree $5$ by proving Theorem~\ref{mainthm-resol}. 

Let $X$ be a del Pezzo $3$-fold of degree $5$ and $\mcE$ a normalized rank $2$ vector bundle. 
It is easy to see that $\mcE$ is a weak Fano bundle if $\mcE$ is one of (i) -- (viii) in Theorem~\ref{mainthm-resol}. 
Hence, it suffices to show that every normalized rank $2$ weak Fano bundle $\mcE$ satisfies one of (i) -- (viii). 

If $\mcE$ is decomposable, then it is easy to see $\mcE$ satisfies one of (i), (ii), or (iii). 
Assume $\mcE$ is indecomposable. 
If $c_{1}(\mcE)=-1$, then Corollary~\ref{cor-c1odd} gives that $\mcE \simeq \mcR$, i.e., $\mcE$ is of type (iv). 
Hence we may assume that $c_{1}(\mcE)=0$. 
Then by Proposition~\ref{prop-Ishikawa}~(3) and (4), we have $1 \leq c_{2}(\mcE) \leq 4$. 
Thus it follows from Theorem~\ref{thm-resol} that $\mcE$ is of type (v), (vi), (vii), or (viii) if $c_{2}(\mcE)=1,2,3,4$ respectively. 

Finally, we show the existence of an example for each of (i) -- (viii). 
This statement is trivial for the cases (i) -- (iv). 
Hence it suffices to show that, for a given $c \in \{1,2,3,4\}$, there exists weak Fano bundle $\mcE$ such that $c_{1}(\mcE)=0$ and $c_{2}(\mcE)=c$. 

When $c=1$, we have such an $\mcE$ as a unique extension of $\mcI_{l/X}$ by $\mcO_{X}$, where $l$ is a line on $X$ (c.f. Proposition~\ref{prop-Ishikawa}~(5)). 
When $c=2$, we have such an $\mcE$ as $\mcF(-1)$, where $\mcF$ is a special Ulrich bundle on $X$, which was constructed in \cite[Proposition~6.1]{Beauville}. 
Hence we may assume that $c \in \{3,4\}$. 
Note that all del Pezzo $3$-folds of degree $5$ are isomorphic \cite{Iskovskikh77}. 
Hence in this case, we can apply \cite[Theorem~5.8]{ACM17} and obtain an elliptic curve $C$ of degree $c+5$ such that $-K_{\Bl_{C}X}$ is nef and big, where $\Bl_{C}X$ is the blowing-up of $X$ along $C$.  
Let $\mcF$ be a unique non-trivial extension of $\mcI_{C/X}(-K_{X})$ by $\mcO_{X}$. 
Then $\Bl_{C}X$ is a member of $|\mcO_{\P(\mcF)}(1)|$ and $-K_{\Bl_{C}X} \simeq \mcO_{\P(\mcF)}(1)|_{\Bl_{C}X}$. 
Thus we can conclude that $\mcO_{\P(\mcF)}(1)$ is nef and big and so is $-K_{\P(\mcF)}$. 
Letting $\mcE:=\mcF(-1)$, we obtain an example of a rank $2$ weak Fano bundle $\mcE$ with $c_{1}(\mcE)=0$ and $c_{2}(\mcE)=c$. 
We complete the proof of Theorem~\ref{mainthm-resol}. \qed

\section{Moduli spaces of rank $2$ weak Fano bundles on a del Pezzo $3$-fold of degree $5$}\label{sec-moduli}

Let $X$ be a del Pezzo $3$-fold of degree $5$. 
As an application of Theorem~\ref{mainthm-resol}, this section studies moduli spaces of rank $2$ weak Fano bundles on $X$. 
First of all, we define the moduli functors that will be studied in this section, by following \cite{HL10}. 

\begin{defi}\label{defi-modulispace}
Let $(\Sch/\Bbbk)$ be the category of schemes of finite type over the base field $\Bbbk$. 
Let $(\Sets)$ be the category of sets. 
For a given $(c_{1},c_{2}) \in \{(0,-1)\} \times \Z$, 
we consider the following functor $\mcM^{\wF}_{c_{1},c_{2}} \colon (\Sch/\Bbbk)^{\op} \to (\Sets)$ defined by 

\[\mcM^{\wF}_{c_{1},c_{2}}(S)=
\left\{ 
\begin{array}{l}
\mcE_{S} \text{ is a coherent sheaf } \\
\text{on } X \times S
\end{array}
\left|
\begin{array}{l}
\mcE_{S} \text{ is flat over } S, \\ 
\forall s \in S, \mcE_{\ol{s}}:=\mcE|_{X_{s}} \otimes_{\Bbbk(s)} \ol{\Bbbk(s)} \text{ is } \\
\text{a rank $2$ weak Fano bundle} \\
\text{on $X_{\ol{s}}=:X_{s} \otimes_{\Bbbk(s)} \ol{\Bbbk(s)}$ s.t. }\\
c_{1}(\mcE_{\ol{s}})=c_{1} \text{ and } c_{2}(\mcE_{\ol{s}})=c_{2} .
\end{array}
\right\}
\right/\sim
,
\]
where $\mcE_{S} \sim \mcE_{S}' \iff \exists \mcL_{S} \in \Pic(S)$ such that $\mcE_{S}' \simeq \mcE_{S} \otimes \pr_{2}^{\ast}\mcL_{S}$. 
In the above definition, note that being weak Fano is preserved under base change for a smooth projective variety over a field (c.f. \cite[Lemma~2.18]{Keeler03}). 

As usual, a scheme $M^{\wF}_{c_{1},c_{2}}$ over $\Bbbk$ is said to be \emph{the coarse (resp. fine) moduli space} of $\mcM^{\wF}_{c_{1},c_{2}}$ if there is a natural transformation $\mcM^{\wF}_{c_{1},c_{2}} \to \Hom(-,M^{\wF}_{c_{1},c_{2}})$ which corepresents (resp. represents) the functor $\mcM^{\wF}_{c_{1},c_{2}}$ and the map $\mcM^{\wF}_{c_{1},c_{2}}(\Spec \Bbbk) \to M^{\wF}_{c_{1},c_{2}}(\Spec \Bbbk)$ is 
an isomorphism. 
\end{defi}


By our classification results obtained so far, 
the determination of the moduli space $M^{\wF}_{c_{1},c_{2}}$ for the case $(c_{1},c_{2}) \not\in \{(0,3),(0,4)\}$ is not so difficult. 
On the other hand, $M^{\wF}_{0,3}$ has already been analyzed by Sanna \cite{Sanna17} since every weak Fano bundle $\mcE$ of rank $2$ with $(c_{1}(\mcE),c_{2}(\mcE))=(0,3)$ is an instanton bundle by Proposition~\ref{prop-FHI}. 
A description of $M^{\wF}_{c_{1},c_{2}}$ with $(c_{1},c_{2}) \neq (0,4)$ is summarized as the following proposition. 

\begin{prop}\label{prop-moduli-prelim}
\begin{enumerate}
\item The coarse moduli spaces $M^{\wF}_{0,-5}$, $M^{\wF}_{-1,0}$, $M^{\wF}_{0,0}$, and $M^{\wF}_{-1,2}$ are isomorphic to the point $\Spec \Bbbk$. 
\item The coarse moduli space $M^{\wF}_{0,1}$ is isomorphic to $\P^{2}$. 
\item The coarse moduli space $M^{\wF}_{0,2}$ (resp. $M^{\wF}_{0,3}$) is a smooth irreducible variety. Moreover, it is not fine (resp. fine) as a moduli space. 
\end{enumerate}
\end{prop}

This proposition can be proved by using known results \cite{FN89,Sanna17} and some fundamental results in the moduli theory.
Since it is slightly out of the focus of this article, the proof for this proposition is postponed until the appendix.

\subsection{Preliminaries and setting}
Our main purpose of this section is to investigate $M^{\wF}_{0,4}$ and prove Theorem~\ref{mainthm-moduli}. 
To show this theorem, we associate this moduli space with the moduli space of representations of a certain quiver. 
Notations about quivers and its representations basically follow those in \cite{King94,Reineke08}. 
Furthermore, this section works with the following additional notations. 




\begin{nota}
\begin{itemize}
\item Let $\mcH$ be the cokernel of the natural map
\[0 \to \mcO_{X}(-1) \to \mcQ(-1) \otimes \Hom(\mcO_{X}(-1),\mcQ(-1))^{\vee} \simeq \mcQ(-1)^{\oplus 5} \to \mcH \to 0.\]
Note that $\mcH=\RR_{\mcQ(-1)}(\mcO_{X}(-1))[1]$ and hence $\braket{\mcO_{X}(-1),\mcQ(-1)}=\braket{\mcQ(-1),\mcH}$. 
\item Set $\mcT:=\mcQ(-1) \oplus \mcH$. 
Since $\Hom(\mcQ(-1),\mcH) \simeq \Bbbk^{\oplus 5}$, $\Hom(\mcH,\mcQ(-1))=0$, and $\mcQ(-1)$ and $\mcH$ are simple, $\End(\mcT)$ is isomorphic to the path algebra $\Bbbk Q$ of the $5$-Kronecker quiver $Q$. 
\item Let $v_{0}$ (resp. $v_{1}$) be the vertex of $Q$ corresponding to $\mcQ(-1)$ (resp. $\mcH$). 
For a representation $M$ of $Q$ over $\Bbbk$, 
the dimension vector $\underline{\dim}(M)$ is defined to be $(\dim_{\Bbbk}M_{0},\dim_{\Bbbk}M_{1})$, 
where we identify $M$ as a right $\Bbbk Q$-module and $M_{i}:=M \cdot e_{v_{i}}$, where $e_{v_{i}}$ is the corresponding idempotent. 
\item We define a stability function $\Theta \colon \Z^{\oplus 2} \to \Z$ as $\Theta(a,b) := b-a$. 
\end{itemize}
\end{nota}

As a preliminary, we prepare the following lemma. 

\begin{prop}\label{lem-quiver-stable}
Let $\mcE$ be a rank $2$ weak Fano bundle with $c_{1}(\mcE)=0$ and $c_{2}(\mcE)=4$. Let $\mcF:=\mcE(1)$ and $\mcK_{\mcF}:=\Ker(H^{0}(\mcF) \otimes \mcO_{X} \to \mcF)$. 
Then the following assertions hold.
\begin{enumerate}
\item $\Ext^{i}(\mcT,\mcK_{\mcF}) = 0$ for every $i>0$. Moreover, $\Hom(\mcT,\mcK_{\mcF})$ is a right $\Bbbk Q \simeq \End(\mcT)$-module of dimension vector $(2,2)$. 
\item $\Hom(\mcT,\mcK_{\mcF})$ is $\Theta$-stable. 
\end{enumerate}
\end{prop}

\begin{proof}
(1) It suffices to show that
$\Ext^{i}(\mcQ(-1),\mcK_{\mcF}) \simeq \Ext^{i}(\mcH,\mcK_{\mcF}) \simeq 
\left\{
\begin{array}{ll}
\Bbbk^{2} & (i=0) \\
0 & (i \neq 0)
\end{array}
\right.
$. 
This directly follows from the exact sequence 
\begin{align}\label{ex-K-4}
0 \to \mcO_{X}(-1)^{\oplus 2} \to \mcQ(-1)^{\oplus 2} \to \mcK_{\mcF} \to 0
\end{align}
which is obtained in Theorem~\ref{mainthm-resol}~(viii). 

(2) 
Since $\braket{\mcQ(-1),\mcH}$ is a strong exceptional collection, 
the functor 
\begin{align}\label{eq-equiv}
\Phi \colon \braket{\mcQ(-1),\mcH} \ni \mcF \to \RHom(\mcT,\mcF) \in \Db(\mod{\mhyphen}\Bbbk Q) 
\end{align}
gives a category equivalence. 
Note that the heart $\{M^{\bullet} \in \Db(\mod{\mhyphen}\Bbbk Q) \mid \mcH^{\neq 0}(M^{\bullet})=0\}$ of the standard $t$-structure of $\Db(\mod{\mhyphen}\Bbbk Q)$ is a full abelian subcategory. 
Let $\mcA \subset \braket{\mcQ(-1),\mcH}$ be the corresponding abelian subcategory under $\Phi$. 
Then $\mcA=\{\mcF^{\bullet} \in \braket{\mcQ(-1),\mcH} \mid \Ext^{\neq 0}(\mcT,\mcF^{\bullet})=0\}$. 
For each object $\mcF^{\bullet} \in \mcA$, 
we define its dimension vector as 
$\underline{\dim}(\mcF^{\bullet})
:=\underline{\dim}(\Phi(\mcF^{\bullet}))
=(\dim \Hom(\mcQ(-1),\mcF^{\bullet}),\dim \Hom(\mcH,\mcF^{\bullet}))$. 

By (1), $\mcK_{\mcF} \in \braket{\mcQ(-1),\mcH}$ is an object of $\mcA$. 
Therefore, to show the $\Theta$-stability, 
it suffices to show the inequality 
$\Theta(M)=\dim \Hom(\mcH,M) - \dim \Hom(\mcQ(-1),M) < 0$ for every subobject $0 \neq M \subsetneq \mcK_{\mcF}$ in $\mcA$. 

Since $\RHom(\mcT,\mcQ(-1))=\Bbbk$ and $\RHom(\mcT,\mcO_{X}(-1))=\Bbbk[-1]$, 
$\mcQ(-1)$ and $\mcO_{X}(-1)[1]$ are objects in $\mcA$. 
Taking a shift of the exact sequence (\ref{ex-K-4}) gives an exact sequence 
\begin{align}\label{ex-K-4-inA}
0 \to \mcQ(-1)^{\oplus 2} \to \mcK_{\mcF} \to \mcO_{X}(-1)[1]^{\oplus 2} \to 0
\end{align}
in $\mcA$. Let $0 \neq M \subsetneq \mcK_{\mcF}$ be a subobject in $\mcA$ and $N:=\mcK_{\mcF}/M \in \mcA$. 
Since $M,N \in \mcA \subset \braket{\mcQ(-1),\mcH}=\braket{\mcO_{X}(-1)[1],\mcH}$, 
Lemma~\ref{lem-2term} gives two distinguished triangles: 
\begin{align*}
&\RHom(\mcQ(-1),M) \otimes \mcQ(-1) \to M \to \RHom(M,\mcO_{X}(-1))^{\vee} \otimes \mcO_{X}(-1) \mathop{\to}^{+1} \text{ and } \\
&\RHom(\mcQ(-1),N) \otimes \mcQ(-1) \to N \to \RHom(N,\mcO_{X}(-1))^{\vee} \otimes \mcO_{X}(-1) \mathop{\to}^{+1}.
\end{align*}
Since 
$\Ext^{\neq 0}(\mcQ(-1),M)
=\Ext^{\neq 0}(\mcQ(-1),N)
=0$ and $\mcO_{X}(-1)[i] \in \mcA$ if and only if $i = 1$, 
there exist $a,b,c,d \in \Z_{\geq 0}$ such that 
$\RHom(\mcQ(-1),M) \simeq \Bbbk^{\oplus a}$, 
$\RHom(M,\mcO_{X}(-1))^{\vee} \simeq \Bbbk^{\oplus b}[1]$, 
$\RHom(\mcQ(-1),N) \simeq \Bbbk^{\oplus c}$, and
$\RHom(N,\mcO_{X}(-1))^{\vee} \simeq \Bbbk^{\oplus d}[1]$, 
and the following diagram in $\mcA$: 
\[\xymatrix{
&0\ar[d]&0\ar[d]&0\ar[d]& \\
0\ar[r]&\mcQ(-1)^{\oplus a} \ar[r] \ar[d]& \mcQ(-1)^{\oplus 2} \ar[r] \ar[d]& \mcQ(-1)^{\oplus c} \ar[r] \ar[d]& 0 \\
0\ar[r]&M \ar[r] \ar[d]& \mcK_{\mcF} \ar[r] \ar[d]& N \ar[r] \ar[d]& 0 \\
0\ar[r]&\mcO_{X}(-1)[1]^{\oplus b} \ar[r] \ar[d] & \mcO_{X}(-1)[1]^{\oplus 2} \ar[r] \ar[d] & \mcO_{X}(-1)[1]^{\oplus d} \ar[r] \ar[d]& 0 \\
&0&0&0.& 
}\]
Thus $c=2-a$ and $d=2-b$. 
Since $\underline{\dim}(\mcQ(-1))=(1,0)$ and $\underline{\dim}(\mcO_{X}(-1)[1])=(0,1)$, 
we have $\underline{\dim}(M)=(a,b)$ and $\Theta(M)=b-a$. 
It suffices to show $\Theta(M)<0$. 

First, it follows from the equality $\Hom_{\mcA}(\mcO_{X}(-1)[1],\mcK_{\mcF}) = \Ext^{-1}_{\Db(X)}(\mcO_{X}(-1),\mcK_{\mcF})=0$ that $a \neq 0$. 
Next, we show $d \neq 0$. 
If $d=0$, then there is a surjection $\mcK_{\mcF} \epm N \simeq \mcQ(-1)^{\oplus c}$. 
However, if there is a surjection $\mcK_{\mcF} \epm \mcQ(-1)$, then the kernel is a line bundle, which must be isomorphic to $\mcO_{X}$, which contradicts that $h^{0}(\mcK_{\mcF})=0$. 
Hence $c=0$, which implies $N=0$. 
This contradicts our assumption $M \subsetneq \mcK_{\mcF}$. 

Hence it holds that $a \neq 0$ and $d \neq 0$, which implies $a \in \{1,2\}$ and $b \in \{0,1\}$. 
To show $\Theta(M) < 0$, it suffices to show that the case $a=b=1$ does not occur. 
If $a=b=1$, then there is an exact sequence 
$\displaystyle 
0 \to \mcQ(-1) \to M \to \mcO_{X}(-1)[1] \to 0$
in $\mcA$. 
Let $s \in \Ext^{1}(\mcO_{X}(-1)[1],\mcQ(-1)) \simeq \Hom_{\Db(X)}(\mcO_{X}(-1),\mcQ(-1))$ be the corresponding extension class. 
Then the mapping cone of $s$ is isomorphic to $M$. 
If $s=0$, then $M \simeq \mcO_{X}(-1)[1] \oplus \mcQ(-1)$, which contradicts that $\Hom(\mcO_{X}(-1)[1],\mcK_{\mcF})=0$. 
Hence $s \neq 0$ and the mapping cone of $s$ is isomorphic to the cokernel of $s$ in $\Coh(X)$.
In particular, $M$ is a coherent sheaf, which is not locally free since $(s=0) \neq \emp$, i.e.~the bundle $\mcQ$ does not have nowhere vanishing section. 
Since $c=d=1$ also holds, it follows from the exactly same argument that $N$ is a coherent sheaf which is not locally free and fits in an exact sequence
\[ 0 \to \mcO_X(-1) \to \mcQ(-1) \to N \to 0 \]
in the coherent heart $\Coh(X)$.
However the argument using depth applied to an exact sequence $0 \to M \to \mcK_{\mcF} \to N \to 0$, where $\mcK_{\mcF}$ is locally free, implies that $M$ must be locally free. 
This is a contradiction, which completes the proof of (2). 
\end{proof}

\subsection{Irreducibility of $M^{\wF}_{0,4}$.}

We define the moduli functor $\mcM^{\Theta\text{-st}}_{(2,2)}(Q) \colon (\Sch/\Bbbk)^{\op} \to (\Sets)$ as follows: 
\[\mcM^{\Theta\text{-st}}_{(2,2)}(Q)(S):=
\left\{ 
\begin{array}{l}
\mcV \text{: a locally free}  \\
\text{sheaf on } S 
\end{array}
\left|
\begin{array}{l}
\mcV \text{ is a family of $\Bbbk Q$-modules } \\
\text{over $S$ with } \underline{\dim}(\mcV)=(2,2), \\
\forall s \in S, \mcV \otimes \Bbbk(s) \text{ is a geometrically} \\
\Theta \text{-stable representation of } Q. 
\end{array}
\right\}
\right./\sim
,
\]
where $\mcV \sim \mcV' \iff \exists \mcL \in \Pic(S)$ such that $\mcV' \simeq \mcV \otimes \mcL$. 
For the definition of families of $\Bbbk Q$-modules over $S$ and its dimension vectors, we refer to \cite{King94}.
In the following proposition, 
we see a relationship between the moduli functors
$\mcM^{\wF}_{0,4}$ and $\mcM^{\Theta\text{-st}}_{(2,2)}(Q)$. 
\begin{prop}\label{prop-moduli-map}
For $\mcE \in \mcM^{\wF}_{0,4}(S)$, we set $\mcF:=\mcE \otimes \pr_{1}^{\ast}\mcO_{X}(1)$ and define 
\begin{align}\label{ex-fam-1}
\mcK_{\mcF}:=\Ker (\pr_{2}^{\ast}{\pr_{2}}_{\ast}\mcF \to \mcF).
\end{align}
Since $\mcF$ and ${\pr_{2}}_{\ast}\mcF$ are locally free
, so is $\mcK_{\mcF}$. 
Note that $\mcK_{\mcF}$ is also flat over $S$. 

Then the following assertions hold. 
\begin{enumerate}
\item ${\mcE}xt_{\pr_{2}}^{i}(\pr_{1}^{\ast}\mcQ(-1),\mcK_{\mcF})={\mcE}xt_{\pr_{2}}^{i}(\pr_{1}^{\ast}\mcH,\mcK_{\mcF})=0$ for $i>0$. 
Moreover, the coherent sheaves $\mcV_{1}:={\pr_{2}}_{\ast}{\mcH}om(\pr_{1}^{\ast}\mcQ(-1),\mcK_{\mcF})$ and $\mcV_{2}:={\pr_{2}}_{\ast}{\mcH}om(\pr_{1}^{\ast}\mcH,\mcK_{\mcF})$ are locally free sheaves on $S$ of rank $2$. 
In particular, $\mcV:=\mcV_{1} \oplus \mcV_{2}$ is a family of $\Bbbk Q$-module. 
\item $\mcV_{1} \oplus \mcV_{2}$ is a geometrically $\Theta$-stable representation of the quiver $Q$ of dimension vector $(2,2)$. 
\end{enumerate}
In particular, the morphism of functors 
\[F(S) \colon \mcM^{\wF}_{0,4}(S) \ni \mcE \mapsto {\pr_{2}}_{\ast}{\mcH}om(\pr_{1}^{\ast}\mcT,\mcK_{\mcE(1)}) = \mcV_{1} \oplus \mcV_{2} \in \mcM^{\Theta\text{-st}}_{(2,2)}(Q) (S)\]
is well-defined. 
\end{prop}

\begin{proof}
Since $\mcK_{\mcF}$ is locally free, it holds that ${\mcE}xt_{\pr_{2}}^{i}(\pr_{1}^{\ast}\mcE_{X},\mcK_{\mcF})
\simeq R^{i}{\pr_{2}}_{\ast}{\mcH}om(\pr_{1}^{\ast}\mcE_{X},\mcK_{\mcF})$ for every locally free sheaf $\mcE_{X}$ on $X$. 
Then the Cohomology and Base Change theorem and Proposition~\ref{lem-quiver-stable}~(1) shows 
${\mcE}xt_{\pr_{2}}^{i}(\pr_{1}^{\ast}\mcQ(-1),\mcK_{\mcF})={\mcE}xt_{\pr_{2}}^{i}(\pr_{1}^{\ast}\mcH,\mcK_{\mcF})=0$ for $i>0$, and the locally freeness of $\mcV_{1}$ and $\mcV_{2}$, which shows (1). 
(2) follows from Proposition~\ref{lem-quiver-stable}~(2). 
\end{proof}

Now we have the following proposition, which is a half part of Theorem~\ref{mainthm-moduli}. 

\begin{prop}\label{prop-irred}
The coarse moduli space $M_{0,4}^{\wF}$ of $\mcM_{0,4}^{\wF}$ is a smooth variety of dimension $13$. 
\end{prop}
\begin{proof}
Let $\mcM_{0,4}$ be the moduli functor of Gieseker-semistable rank $2$ coherent sheaves with $c_{1}=0$ and $c_{2}=4$ on $X$. 
By Propositions~\ref{prop-Ishikawa} and \ref{prop-FHI} and Remark~\ref{rem-st}, 
for a rank $2$ vector bundle $\mcE$ with $c_{1}(\mcE)=0$ and $c_{2}(\mcE)=4$, $\mcE$ is weak Fano if and only if $\mcE$ is stable and $\mcE(1)$ is globally generated.
Hence the natural map $\mcM_{0,4}^{\wF} \to \mcM_{0,4}$ is well-defined and an open immersion. 
Thus there exists the coarse moduli space $M:=M_{0,4}^{\wF}$ as an open subscheme of the coarse moduli space $M_{0,4}$ of $\mcM_{0,4}$, which is constructed by \cite[Theorem 4.3.4]{HL10}.


To show the smoothness of $M$ and $\dim M=13$, 
it suffices to check that $\dim \Ext^{1}(\mcF,\mcF)=13$ and $\Ext^{2}(\mcF,\mcF)=0$ for each $[\mcF] \in M(\Bbbk)$ \cite[Corollary~4.5.2]{HL10}. 
Let $s \in H^{0}(\mcF)$ be a general section. 
Then $C=(s=0)$ is an elliptic curve on $X$ of degree $9$. 
There is an exact sequence $0 \to \mcO_{X} \to \mcF \to \mcI_{C/X}(-K_{X}) \to 0$ which gives 
\[0 \to \mcF^{\vee} \to \mcF \otimes \mcF^{\vee} \to \mcF^{\vee} \otimes \mcI_{C/X}(2) \to 0.\]
Since $\RG (X,\mcF^{\vee})=\RG(X,\mcE(-1))=0$ by Proposition~\ref{prop-FHI}, 
it holds that $\RHom(\mcF,\mcF) \simeq \RG(\mcF^{\vee} \otimes \mcI_{C/X}(2)) \simeq \RG(\mcF \otimes \mcI_{C/X})$, 
where we identify $\mcF \simeq \mcF^{\vee}(2)$ given by the perfect pairing $\mcF \otimes \mcF \to \det \mcF \simeq \mcO_{X}(2)$.
Consider 
\[0 \to \mcI_{C/X} \otimes \mcF \to \mcF \to \mcF|_{C}  \to 0.\]
Note that $\mcF|_{C} \simeq \mcN_{C/X}$. 
Moreover, by the classification of weak Fano $3$-folds of Picard rank $2$ whose crepant contraction is divisorial \cite{JPR05}, it is known that the crepant contraction from the weak Fano $3$-fold $\Bl_{C}X$ is small. 
Thus $\mcN_{C/X} \simeq \mcN_{C/X}^{\vee} \otimes \mcO_{X}(-K_{X})|_{C}$ has no trivial line bundle as its quotient. 
Then the Serre duality gives $H^{1}(\mcN_{C/X})=H^{0}(\mcN_{C/X}^{\vee})^{\vee}=0$ and 
the Riemann-Roch theorem gives 
$\RG (\mcF|_{C}) \simeq \RG(\mcN_{C/X}) \simeq \Bbbk^{\oplus 18}$. 
Note that $\RG(\mcF) \simeq \Bbbk^{\oplus 6}$ by Proposition~\ref{prop-Ishikawa}~(1) and (2). 
Since $H^{0}(\mcF \otimes \mcI_{C/X}) \simeq \Hom(\mcF,\mcF)=\Bbbk$, 
we have $\dim \Ext^{1}(\mcF,\mcF)=h^{1}(\mcF \otimes \mcI_{C/X})=13$ and $\dim \Ext^{2}(\mcF,\mcF)=h^{2}(\mcF \otimes \mcI_{C/X})=0$.

On the other hand, It was known that the moduli functor $\mcM^{\Theta\text{-st}}_{(2,2)}(Q)$ is corepresented by a smooth irreducible variety $M':=M^{\Theta\text{-st}}_{(2,2)}(Q)$ 
(\cite[Section~3.5]{Reineke08}, \cite[Proposition~5.2]{King94}). 
Let $f \colon M \to M'$ be the morphism induced by the functor in Proposition~\ref{prop-moduli-map}. 
Now it suffices to show that $f$ is an open immersion. 
Since $M$ and $M'$ are smooth and $\dim M=\dim M'=13$, 
it is enough to check that $f(\Bbbk) \colon M(\Bbbk)=\mcM(\Bbbk) \to M'(\Bbbk)=\mcM'(\Bbbk)$ is injective. 
Let $\mcE_{1},\mcE_{2}$ be rank $2$ weak Fano bundles with $c_{1}(\mcE_{i})=0$ and $c_{2}(\mcE_{i})=4$. 
Set $\mcF_{i}:=\mcE_{i}(1)$ for each $i$. 
It is enough to show that 
$\mcF_{1} \simeq \mcF_{2}$ if and only if 
$\Hom(\mcT,\mcK_{\mcF_{1}}) \simeq \Hom(\mcT,\mcK_{\mcF_{2}})$ as representation of $Q$, where $\mcK_{\mcF_{i}}:=\Ker(H^{0}(\mcF_{i}) \otimes \mcO_{X} \to \mcF_{i})$. 
Note that $H^{0}(\mcF_{i})^{\vee} \simeq H^{0}(\mcK_{\mcF_{i}}^{\vee})$ 
follows from the exact sequence $0 \to \mcF_{i}^{\vee} \to H^{0}(\mcF_{i})^{\vee} \otimes \mcO_{X} \to \mcK_{\mcF_{i}}^{\vee} \to 0$ and 
$H^{1}(\mcF_{i}^{\vee})=H^{1}(\mcE_{i}(-1))=0$. 
Thus $\mcF_{1} \simeq \mcF_{2}$ if and only if $\mcK_{\mcF_{1}} \simeq \mcK_{\mcF_{2}}$. 
By the equivalence $\Phi$ in (\ref{eq-equiv}), 
$\mcK_{\mcF_{1}} \simeq \mcK_{\mcF_{2}}$ if and only if $\Hom(\mcT,\mcK_{\mcF_{1}}) \simeq \Hom(\mcT,\mcK_{\mcF_{2}})$ as $\End(\mcT)=\Bbbk Q$-modules. 
Therefore, $f(\Bbbk)$ is injective, which concludes that $f \colon M \to M'$ is an open immersion. 
Hence we conclude that $M$ is irreducible since so is $M'$. 
This completes the proof.
\end{proof}

Now Theorem~\ref{mainthm-moduli} is a direct consequence of \cite{RS17}. 

\begin{proof}[Proof of Theorem~\ref{mainthm-moduli}]
Since $M'=M^{\Theta\text{-st}}_{(2,2)}(Q)$ is the moduli of stable representations of the $5$-Kronecker quiver $Q$ whose dimension vector is $(2,2)$,  \cite[Theorems~4.4 and 6.1]{RS17} gives that $\Br(M') \simeq \Z/2\Z$ and the Brauer-Severi scheme $P_{i} \to M'$ defined in \cite[P.457, l.-9 -- l.-4]{RS17} is a generator of this Brauer group, where $i \in Q_{0}=\{v_{0},v_{1}\}$. 
If $M$ is fine as a moduli space, then by the functor in Proposition~\ref{prop-moduli-map}, there is a universal family $\mcV_{1,M} \oplus \mcV_{2,M}$ of $\Theta$-stable $\Bbbk Q$-module over $M$ with dimension vector $(2,2)$. 
Then the projective bundle $P_{i,M}:=\P_{M}(\mcV_{i,M})$ gives a trivial Brauer class, which is isomorphic to the base change of $P_{i}$. 
This contradicts the injectivity of the restriction map $\Br(M') \to \Br(M)$ (c.f. \cite[Proposition~2.3]{RS17}). 
\end{proof}

\subsection{Geometric aspects of $M^{\wF}_{0,4}$.}
In this section, we see some geometric property of the coarse moduli space $M^{\wF}_{0,4}$. 
From now on, fix a vector space $V=\Bbbk^{5}$.
Recall the following varieties for $r \in \{1,2,3,4\}$, which were introduced in \cite{HT15}; 
\begin{align*}
S_{r}&:=\{ Q \in |\mcO_{\P(V)}(2)| \mid \rk Q \leq r\} \subset |\mcO_{\P(V)}(2)| = \P(\Sym^{2}V^{\vee})  \text{ and } \\
U_{r}&:= \{([\Pi],Q) \in \Gr(2,V) \times S_{r} \mid \Pi \subset Q\}.
\end{align*}
In particular, $S_{4}$ parametrizes the singular hyperquadrics on $\P(V)$. 
In the above definition, we regard $\Gr(2,V)$ as the parametrizing space of $2$-planes in $\P(V):=\Proj \Sym^{\bullet} V$. 
Note that $U_{4} \simeq \P_{\Gr(2,V)}(\mcE)$, where $\mcE:=\Ker(\Sym^{2}V \otimes \mcO_{\Gr(2,V)} \to \Sym^{2}\mcQ)^{\vee}$. 
Let 
$\displaystyle U_{4} \mathop{\to}^{\pi} T_{4} \mathop{\to}^{\rho} S_{4}$ be the Stein factorization as summarized in the following diagram.
\begin{align}\label{diag-HT}
\xymatrix{
&\ar[ld]_{p}U_{4}\ar[rd]^{q} \ar[d]_{\pi}& \\
\Gr(2,V)&T_{4}\ar[r]_{\rho}&S_{4}
}
\end{align}
It was proved by \cite[Proposition 2.3]{HT15} that $T_{4} \to S_{4}$ is a double covering branched along $S_{3} \subset S_{4}$. 

By \cite[Proposition~4.10]{Chung-Moon17}, it was proved that $T_{4}$ contains the moduli space $M'$ as an open subscheme. 
More precisely, $T_{4}$ is the moduli space $M^{\Theta\text{-sst}}_{(2,2)}(Q)$ of the semi-stable representations of the $5$-Kronecker quiver whose dimension vector is $(2,2)$. 
Then by the proof of Proposition~\ref{prop-irred}, 
$M$ is naturally contained in $T_{4}$ as an open subscheme. 

This geometric observation enables us to find an alternative proof of Theorem~\ref{mainthm-moduli}, which goes as follows. 
Set $U_{M}:=\pi^{-1}(M)$, $\pi_{M}:=\pi|_{U_{M}} \colon U_{M} \to M$, and $p_{M} \colon U_{M} \hra U_{4} \to \Gr(2,V)$. 
Consider $F_{M}:=U_{M} \times_{\Gr(2,V)} \Fl(2,4;V)$ and let $Q_{M}$ be the image of the natural morphism 
$F_{M} \to  M \times \P(V)$:
\begin{align}\label{diag-flag}
\xymatrix{
&&\Fl(2,4;V)\ar[lldd] \ar[rrdd] &\\
&\ar@{}[d]|{\Box}& \ar[u] \P_{U_{M}}(p_{M}^{\ast}\mcQ)=F_{M} \ar[ld]_{f} \ar[rd]^{\t}&\\
\Gr(2,V) &\ar[l]_{p_{M}}U_{M}\ar[rd]_{\pi_{M}}& &Q_{M} \ar[ld]_{q} \ar[r]& \P(V)\\
&&M.&&
}
\end{align}
For each $x:=[A_{x} \colon \Bbbk^{2} \to \Bbbk^{2} \otimes V] \in M$, 
the corresponding quadric $\rho(x) = [Q_{x}] \in S_{4}$ is given by the determinant form $\det A_{x} \in \Sym^{2}V$. 
Since $x$ is stable, $Q_{x} \subset \P(V)$ is of rank $4$ by \cite[Lemma~4.3]{Chung-Moon17}. 
Then $C_{x}$ is a smooth conic on $\Gr(2,V)$ which parametrizes a one component of the Hilbert scheme of $2$-planes in $Q_{x}$. 
The universal family $f^{-1}(C_{x}) \to C_{x}$ can be identified with $\P_{\P^{1}}(\mcO_{\P^{1}} \to \mcO_{\P^{1}}(1)^{\oplus 1}) \to \P^{1}$ and 
the morphism $f^{-1}(C_{x}) \simeq \P_{\P^{1}}(\mcO_{\P^{1}} \to \mcO_{\P^{1}}(1)^{\oplus 1}) \to Q_{x}$ can be identified with the morphism given by the tautological bundle.

Let us assume that $M$ is fine to obtain a contradiction. 
From now on, denote a coherent sheaf $\pr_{1}^{\ast}\mcF \otimes \pr_{2}^{\ast}\mcG$ on a direct product by $\mcF \boxtimes \mcG$. 
Let $\mcE_{M} \in \mcM(M)$ be the universal object on $X \times M$ and $\mcF_{M}:=\mcE_{M} \otimes \pr_{1}^{\ast}\mcO_{X}(1)$. 
As Proposition~\ref{prop-moduli-map}, there exists rank $2$ vector bundles $\mcV_{1}$ and $\mcV_{2}$ on $M$ and the following exact sequence on $X \times M$. 
\[0 \to \mcO_{X}(-1) \boxtimes \mcV_{2} \mathop{\to}^{\alpha}  \mcQ(-1) \boxtimes \mcV_{1} \to \mcK_{M} \to 0.\]
This $\alpha$ gives an injective morphism ${\pr_{2}}_{\ast}(\alpha  \otimes \pr_{1}^{\ast}\mcO_{X}(1)) \colon \mcV_{2} \to  V \otimes \mcV_{1}$ on $M$. 
Let $A \in \Hom_{M}(\mcV_{2},\mcV_{1}) \otimes V$ be the corresponding element, which induces 
\[0 \to \mcO_{\P(V)}(-1) \boxtimes \mcV_{2} \mathop{\to}^{A} \mcO_{\P(V)} \boxtimes \mcV_{1} \to \mcL_{M}  \to 0\]
on $\P(V) \times M$. 
Note that every $x \in M$, 
$A_{x} \colon \mcO_{\P(V)}(-1) \otimes (\mcV_{2} \otimes \Bbbk(x)) \to \mcO_{\P(V)} \otimes (\mcV_{1} \otimes \Bbbk(x))$ 
is nothing but the representation. 
Hence the support of $\mcL_{M}$ coincides with $Q_{M}$ and $\mcL_{M}$ is a divisorial sheaf on $Q_{M}$. 
Let $D$ be a Weil divisor on $Q_{M}$ such that $\mcO_{Q_{M}}(D)=\mcL_{M}$ and $\wt{D}$ the proper transform of $D$ on $F_{M}$ via the birational morphism $F_{M} \to Q_{M}$. 
For each $x \in M$, $\mcL_{x}$ on $Q_{x}$ is the divisorial sheaf corresponding to a $2$-plane in $Q_{x}$. 
Hence there is a Cartier divisor $E$ on $U_{M}$ such that $f^{\ast}E \sim \wt{D}$. 
Note that every fiber $f$ of $\pi_{M}$ is a smooth rational curve and $E.f=1$. 

Taking the closure of $E$, 
we obtain a divisor $\ol{E}$ on $U_{4}$. 
Recall the morphisms $p$ and $\pi$ as in the diagram (\ref{diag-HT}).  
Since $U_{4}$ is the projectivization of a vector bundle $\mcE$ on $\Gr(2,5)$, $\Pic(U_{4}) \simeq \Z[H] \oplus \Z[\xi]$, where $H$ is the pull-back of a hyperplane section of $\Gr(2,V)$ under $p \colon U_{4} \to \Gr(2,V)$ and $\xi$ a tautological divisor.  
Then there exist $a,b \in \Z$ such that $\ol{E} \sim aH+b\xi$. 
Since a general fiber $C$ of $\pi \colon U_{4} \to T_{4}$ is a conic of $\Gr(2,5)$ and $\ol{E}.C=1$, 
we have $a=(1/2)$, which is a contradiction. 
Hence the moduli space $M$ never be the fine moduli space. 
In both proofs, the non-triviality of the Brauer group of $T_{4}$ is the most essential point. 

\section{Conclusions}\label{sec-conclusions}

With the previous work \cite{Ishikawa16} and \cite{FHI20} and our proof of Theorem~\ref{mainthm-resol}, we have completed the classification of rank $2$ weak Fano bundles on a del Pezzo $3$-fold of degree $d \in \{3,4,5\}$. 
In this concluding section shows that, on del Pezzo $3$-fold of degree $d \in \{1,2\}$, every rank $2$ weak Fano bundle splits, which is namely Theorem~\ref{mainthm-Ishikawa}. 

\begin{proof}[Proof of Theorem~\ref{mainthm-Ishikawa}]
Let $X$ be a del Pezzo $3$-fold of degree $d<3$ and $\mcE$ a weak Fano bundle of rank $2$. 
By Proposition~\ref{prop-Ishikawa}~(4) and (5), it suffices to show $c_{2}(\mcE) < 2$. 
By Proposition~\ref{prop-Ishikawa}~(3), if $c_{2} \geq 2$, then $c_{2}=2$, $c_{1}=0$ and $d=2$. 
By Proposition~\ref{prop-Ishikawa}~(1) and (2), we have $h^{0}(\mcE(1)) \geq \chi(\mcE(1))=2$. 
Let $\pi_{\mcE} \colon \P(\mcE) \to X$ be the projectivization, $\xi_{\mcE}$  a tautological divisor, and $H_{X}$ a hyperplane section on $X$. 
Then the divisor $\xi_{\mcE}+\pi_{\mcE}^{\ast}H_{X}$ is linearly equivalent to an effective divisor 
since $h^{0}(\mcO_{\P(\mcE)}(\xi_{\mcE}+\pi_{\mcE}^{\ast}H_{X}))=h^{0}(\mcE(1))>0$. 
Since $-K_{\P(\mcE)} \sim 2\xi_{\mcE}+\pi_{\mcE}^{\ast}(3H_{X})$ is nef and big, 
we obtain $0 \leq (-K_{\P(\mcE)})^{3}(\xi_{\mcE}+\pi^{\ast}H_{X})=-2$, which is a contradiction. 
\end{proof}

Now we complete our classification of rank $2$ weak Fano bundles on a del Pezzo $3$-fold $X$ of Picard rank $1$ by \cite{Ishikawa16}, \cite{FHI20}, and this article. 
Let us compare the classification of the Fano bundles \cite{mos2} with ours.
A \emph{Fano bundle} $\mcF$ on $X$ are defined to be a vector bundle on $X$ such that $-K_{\P_{X}(\mcF)}$ is ample and its classification is given by \cite{mos2} when $\mcF$ is of rank $2$. 
Compared to their classification, our classification implies that every rank $2$ indecomposable weak Fano bundle $\mcF$ with $c_{1}(\mcF) \in \{1,2\}$ satisfies either of the following two conditions. 
\begin{itemize}
\item When $\mcF$ is a Fano bundle, $c_{1}(\mcF)=1$ and $\mcF$ is globally generated. This follows from \cite{mos2}.
\item When $\mcF$ is a weak Fano bundle but not a Fano bundle, $c_{1}(\mcF)=2$ and $\mcF$ is globally generated. This follows from our classification.
\end{itemize}
In other words, rank $2$ indecomposable weak Fano bundles are determined to be Fano or not by the parity of $c_{1}(\mcF)$. 
This property is so far only known from the classification.
As a summarized result, we can see that either of the following conditions hold for every normalized rank $2$ weak Fano bundle $\mcE$ on a del Pezzo $3$-fold $X$ of degree $d \leq 5$. 
\begin{itemize}
\item $\mcE$ is a direct sum of line bundles, 
\item $\deg X \geq 4$ and $\mcE$ is a indecomposable Fano bundle with $c_{1}(\mcE)=-1$, 
\item $\deg X \geq 3$ and $\mcE$ fits into an exact sequence $0 \to \mcO_{X} \to \mcE \to \mcI_{l/X} \to 0$, where $l \subset X$ is a line, 
\item $\deg X \geq 3$ and $\mcE$ is a minimal instanton bundle, which is equivalent to saying that $\mcE(1)$ is a special Ulrich bundle \cite{Beauville}, or 
\item $\deg X \geq 4$ and $\mcE$ fits into an exact sequence 
$0 \to \mcO_{X}(-1) \to \mcE \to \mcI_{C/X}(1) \to 0$, where 
$C \subset X \subset \P^{\deg X+1}$ is a non-degenerated elliptic curve defined by quadratic equations in $\P^{\deg X+1}$ with $\deg X+3 \leq \deg C < 2\deg X$. 
\end{itemize}

\appendix

\section{Proof of Proposition~\ref{prop-moduli-prelim}}

Let $X$ be a del Pezzo $3$-fold of degree $5$ over $\Bbbk$. 
To prove Proposition~\ref{prop-moduli-prelim}, we prepare the following lemma. 

\begin{lem}
\begin{enumerate}
\item[(1)] For a line $l \subset X$, $\Ext^1(\mcI_{l/X}, \mcO_X) = \Bbbk$.
In particular the exact sequence 
\begin{align} \label{ishikawa exact seq}
 0 \to \mcO_X \to \mcE \to \mcI_{l/X} \to 0 
\end{align}
 in Proposition~\ref{prop-Ishikawa}~(5) is uniquely determined by $\mcI_{l/X}$.
\item[(2)] Let $\mcE$ be a rank two weak Fano bundle on $X$ with $c_1(\mcE) = 0$ and $c_2(\mcE) = 1$. Then $H^0(\mcE) = \Bbbk$.
In particular, the ideal sheaf $\mcI_{l/X}$ of a line $l$ that fits in (\ref{ishikawa exact seq}) is uniquely determined by $\mcE$.
\end{enumerate}
\end{lem}

\begin{proof}
(1) follows from the exact sequence $0 \to \mcI_{l/X} \to \mcO_X \to \mcO_l \to 0$ and the Serre duality $\Ext^1(\mcI_{l/X}, \mcO_X) \simeq H^2(\mcI_{l/X}(-2))^{\vee} \simeq H^{1}(\mcO_{\P^{1}}(-2))^{\vee} \simeq \Bbbk$. 
(2) follows from (\ref{ishikawa exact seq}).
\end{proof}

The following criterion is useful to see whether moduli functors are corepresentable.
\begin{lem}\label{lem-criteria-corep}
Let $\mcM \colon (\Sch/\Bbbk)^{\op} \to (\Sets)$ be a contravariant functor and $M \in (\Sch/\Bbbk)$ an object. 
Let $\pi \colon \mcM \to \Hom(-,M)$ and $\sigma \colon \Hom(-,M) \to \mcM$ be morphisms such that $\pi \circ \sigma = \id_{\Hom(-,M)}$. 
Suppose that, 
for every $S \in (\Sch/\Bbbk)$ and $[\mcE_{S}] \in \mcM(S)$, 
there exists an open covering $S=\bigcup_{i}S_{i}$ and $g_{i} \in \Hom(S_{i},M)$ such that $[\mcE_{S}|_{S_{i}}]=\s_{i}(g_{i})$, where $\s_{i}:=\s_{S_{i}} \colon \Hom(S_{i},M) \to \mcM(S_{i})$. 
Then $\mcM$ is corepresentable by $M$. 
\end{lem}
\begin{proof}
Let $\tau \colon \mcM \to \Hom(-,T)$ be an arbitrary morphism. 
It suffices to show that $\tau \circ \pi \circ \sigma = \tau \colon \mcM \to \Hom(-,T)$. 
Let $S \in (\Sch/\Bbbk)$ be an arbitrary object and set morphisms as follows: 
\[\xymatrix{
\Hom(S,M) \ar@(l,l)[d]_{\sigma_{S}} \ar[r] & \prod_{i}\Hom(S_{i},M) \ar@(l,l)[d]_{\prod \sigma_{i}} \\ 
\mcM(S) \ar[u]_{\pi_{S}} \ar[d]^{\tau_{S}} \ar[r] & \prod_{i} \mcM(S_{i}) \ar[u]_{\prod \pi_{i}} \ar[d]^{\prod \tau_{i}} \\ 
\Hom(S,T)\ar[r]&\prod_{i}\Hom(S_{i},T) 
}\]

Fix $[\mcE_{S}] \in \mcM(S)$ and set $g:=\pi_{S}([\mcE_{S}])$. 
Take an open covering $S=\bigcup_{i} S_{i}$ and morphisms $g_{i} \colon S_{i} \to M$ as the above assumption. 
Then $g_{i}=(\pi_{i} \circ \s_{i})(g_{i})=\pi_{i}([\mcE_{S}]|_{S_{i}})=g|_{S_{i}}$. 
Let $[\mcE_{g}]:=\sigma_{S}(g)$. 
Since $[\mcE_{g}|_{S_{i}}]=\sigma_{i}(g_{i})=[\mcE_{S}|_{S_{i}}]$,  
the both of $[\mcE_{g}]$ and $[\mcE_{S}]$ go to the same element of $\Hom(S,T)$ under the map $\tau_{S}$. 
Hence $\tau_{S}([\mcE_{S}])=\tau_{S}([\mcE_{g}])=(\tau_{S} \circ \sigma_{S})(g)=(\tau_{S} \circ \sigma_{S} \circ \pi_{S})([\mcE_{S}])$, which shows this lemma. 
\end{proof}

\begin{proof}[Proof of Proposition~\ref{prop-moduli-prelim}]
(1) Let $(c_{1},c_{2}) = (0,-5),(-1,0),(0,0),(-1,2)$ and $\mcE_{0}:=\mcO_{X}(-1) \oplus \mcO_{X}(1),\mcO_{X}(-1) \oplus \mcO_{X}, \mcO_{X}^{\oplus 2},\mcR$ respectively. 
Note that $\mcM^{\wF}_{c_{1},c_{2}}(\Spec \Bbbk)=\{[\mcE_{0}]\}$ as a set. 
Let $\pi \colon \mcM^{\wF}_{c_{1},c_{2}} \to \Hom_{k}(-,\Spec \Bbbk)$ be the natural morphism. 
Let $\sigma \colon \Hom_{k}(-,\Spec \Bbbk) \to \mcM^{\wF}_{c_{1},c_{2}}$ be the morphism defined by giving $\pr_{1}^{\ast}\mcE_{0}$. 
For $S \in (\Sch/\Bbbk)$ and $[\mcE_{S}] \in \mcM(S)$, 
if there is an open covering $S=\bigcup_{i}S_{i}$ such that $\mcE_{S_{i}}:=\mcE_{S}|_{X \times S_{i}} \simeq \pr_{1}^{\ast}\mcE_{0}$, 
then Lemma~\ref{lem-criteria-corep} shows the corepresentability of $\Spec \Bbbk$. 

When $(c_{1},c_{2}) = (0,-5)$ or $(-1,0)$, 
then $\mcE_{\ol{s}} \simeq \mcO_{X_{\ol{s}}}(a_{1}) \oplus \mcO_{X_{\ol{s}}}(a_{2})$ for every $s \in S$ where $(a_{1},a_{2})=(-1,1)$ or $(-1,0)$ respectively. 
Then it follows from the Cohomology and Base Change theorem that there are two invertible sheaves $\mcL_{1},\mcL_{2}$ on $S$ such that 
$\mcE_{S} \simeq \bigoplus_{i=1,2} \pr_{1}^{\ast} \mcO_{X}(a_{i}) \otimes \pr_{2}^{\ast}\mcL_{i}$. 
Taking a local trivialization, we obtain an open covering $S=\bigcup_{i}S_{i}$ such that $\mcE_{S}|_{X \times S_{i}} \simeq \bigoplus_{i=1,2} \pr_{1}^{\ast} \mcO_{X}(a_{i})=\pr_{1}^{\ast}\mcE_{0}$. 
When $(c_{1},c_{2})=(0,0)$ (resp. $(-1,2)$), 
then $\mcE_{\ol{s}} \simeq \mcO_{X_{\ol{s}}}^{\oplus 2}$ (resp. $\mcR$) for every $s \in S$. 
Again by the Cohomology and Base Change theorem, 
there exists a rank $2$ (resp. $1$) locally free sheaf $\mcE$ on $S$ such that $\mcE_{S} \simeq \pr_{2}^{\ast}\mcE$ (resp. $\pr_{1}^{\ast}\mcR \otimes \pr_{2}^{\ast}\mcE$). 
Taking a local trivialization, we obtain an open covering $S=\bigcup_{i}S_{i}$ such that $\mcE_{S}|_{X \times S_{i}} \simeq \pr_{1}^{\ast}\mcE_{0}$. This completes the proof of (1). 

(2) First, we construct a morphism $\tau \colon \mcM^{\wF}_{0,1} \to \Hom(-,\P^{2})$. 
Let $S$ be a scheme of finite type over $\Bbbk$ and $\mcE_{S} \in \mcM^{\wF}_{0,1}(S)$ an object. 
Then $\e \colon \pr_{2}^{\ast}{\pr_{2}}_{\ast}\mcE_{S} \to \mcE_{S}$ is injective and $\Cok \e$ is flat over $S$. 
Moreover, $\mcL_{1}:={\pr_{2}}_{\ast}\mcE_{S}$ and $\mcL_{2}:={\pr_{2}}_{\ast}\mcH{om}(\Cok \e,\mcO_{X \times S})$ are invertible by Cohomology and Base Change. 
Under the isomorphisms 
$H^{0}(\mcO_{S}) \simeq 
H^{0}(\mcL_{2}^{-1} \otimes {\pr_{2}}_{\ast}\mcH{om}(\Cok \e,\mcO_{X \times S}))
=H^{0}(\pr_{2}^{\ast}\mcL_{2}^{-1} \otimes \mcH{om}(\Cok \e,\mcO_{X \times S}))
=\Hom(\Cok \e \otimes \pr_{2}^{\ast}\mcL_{2},\mcO_{X \times S})$, 
there is a morphism $t \colon \Cok \e \otimes \pr_{2}^{\ast}\mcL_{2} \to \mcO$ corresponding to the unit of $H^{0}(S,\mcO_{S})$. 
Since $\Cok \e$ is flat over $S$, $t$ is injective and its cokernel, say $\mcO_{Z}$ where $Z$ denotes a closed subscheme of $X \times S$, is flat over $S$. 
Hence there is an exact sequence $0 \to \pr_{2}^{\ast}(\mcL_{1} \otimes \mcL_{2}) \to \mcE_{S} \otimes \pr_{2}^{\ast}\mcL_{2} \to \mcI_{Z} \to 0$. 
Replacing $\mcE_{S}$ to $\mcE_{S} \otimes \pr_{2}^{\ast}\mcL_{2}$ and putting $\mcL_{S}:=\mcL_{1} \otimes \mcL_{2}$, 
we obtain the following exact sequence
\begin{align}
0 \to \pr_{2}^{\ast}\mcL_{S} \to \mcE_{S} \to \mcI_{Z/X} \to 0.\label{ex-moduli-01}
\end{align}
It is easy to see that $Z_{\ol{s}}$ is a line on $X_{\ol{s}}$ for every $s \in S$. 
Now let us consider the following spectral sequence: 
\begin{align*}
&H^{i}(S,{\mcE}xt^{j}_{\pr_{2}}(\mcI_{Z},\pr_{2}^{\ast}\mcL_{S})) \ra \Ext^{i+j}_{X \times S}(\mcI_{Z},\pr_{2}^{\ast}\mcL_{S}),
\end{align*}
where ${\mcE}xt^{j}_{\pr_{2}}(\mcF,-)$ denotes the $j$-th cohomology of the left exact functor ${\pr_{2}}_{\ast}{\mcH}{om}(\mcF,-) \colon \Coh(X \times S) \to \Coh(S)$ \cite{Lange83}. 
Since ${\mcE}xt^{0}_{\pr_{2}}(\mcI_{Z},\pr_{2}^{\ast}\mcL_{S}) = \mcL_{S}$, 
we have an exact sequence 
\begin{align}
0 \to H^{1}(\mcL_{S}) \to \Ext^{1}(\mcI_{Z},\pr_{2}^{\ast}\mcL_{S}) \mathop{\to}^{\beta} H^{0}({\mcE}xt^{1}_{\pr_{2}}(\mcI_{Z},\mcL_{S})) \to H^{2}(\mcL_{S}) \label{ex-relspect}
\end{align}
Note that ${\mcE}xt^{1}_{\pr_{2}}(\mcI_{Z},\mcL_{S})$ is an invertible sheaf (c.f. \cite[Theorem~1.4]{Lange83}). 
Moreover, the global section $\e:=\delta([\mcE_{S}])$ is a nowhere vanishing section. 
Hence ${\mcE}xt^{1}_{\pr_{2}}(\mcI_{Z},\mcL_{S}) \simeq \mcO_{S}$, which implies $\mcL_{S} \simeq {\mcE}xt^{1}_{\pr_{2}}(\mcI_{Z},\mcO_{S})^{\vee}$. 
By this argument, we have a natural identification
\[\mcM^{\wF}_{0,1}(S) =
\left\{ 
(f \colon S \to \Hilb_{t+1}(X),[\mcE_{S}])
\left|
\begin{array}{l}
[\mcE_{S}] \in \Ext^{1}(\mcI_{Z/X \times S},\pr_{2}^{\ast}\mcL_{S}), \text{ where}\\
 Z:=S \times_{\Hilb_{t+1}(X)} \Univ_{t+1}(X) \text{ and } \\
\mcL_{S}:=\mcE{xt}^{1}_{\pr_{2}}(\mcI_{Z/X \times S},\mcO_{X \times S})^{\vee} \\
\text{s.t. $[\mcE_{S}]$ goes to a unit under the map} \\
\beta \colon \Ext^{1}(\mcI_{Z/X \times S},\pr_{2}^{\ast}\mcL_{S}) \to H^{0}(S,\mcO_{S}). \\
\end{array}
\right\}
\right/\sim
,\]
where $(f,[\mcE_{S}]) \sim (f',[\mcE_{S}']) :\iff f'=f$ and $\mcE_{S}$ is isomorphic to $\mcE_{S}'$ as a $\mcO_{X \times S}$-module. 
In particular, there is a functor $\t \colon \mcM^{\wF}_{0,1} \to \Hom(-,\Hilb_{t+1}(X)) \simeq \Hom(-,\P^{2})$ (c.f. Theorem~\ref{thm-FN}). 

Let us show that $\tau$ is the corepresentation of the functor $\mcM^{\wF}_{0,1}$. 
Let $B$ be an arbitrary $\Bbbk$-scheme and $\alpha \colon \mcM^{\wF}_{0,1} \to \Hom(-,B)$ a functor. 
It suffices to show that there is a morphism $G \colon \Hom(-,\P^{2}) \to \Hom(-,B)$ such that following diagram commutes: 
\[\xymatrix{
\mcM^{\wF}_{0,1} \ar[rr]^{\alpha} \ar[rd]_{\tau}&& \Hom(-,B) \\
&\Hom(-,\P^{2}). \ar@{.>}[ru]_{G}&
}\]
Let $(\AffSch/\Bbbk)$ be the category of affine schemes of finite type over $\Bbbk$. 
First, we make a natural transformation $G \colon \Hom(-,\P^{2}) \to \Hom(-,B)$ for the functors $(\AffSch/\Bbbk)^{\op} \to \Sets$. 
For each affine scheme $U$ of finite type over $\Bbbk$,  
let $f_{U} \colon U \to \P^{2}$ be a morphism. 
Since $\mcM^{\wF}_{0,1}(U)=(\Hom(U,\P^{2}) \times \mcO(U)^{\times})/\sim$, 
There is a lift $[f_{U},e_{U}] \in \mcM^{\wF}_{0,1}(U)$ is a lift of $f_{U} \in \Hom(U,\P^{2})$, where $e_{U} \in \mcO(U)^{\times}$ is the multiplicative identity of the ring $\mcO(U)$. 
Let $G_{U}(f_{U})$ be the morphism $\tau_{U}([f_{U},e_{U}]) \colon U \to B$. 
From this construction, $G_{U} \colon \Hom(U,\P^{2}) \to \Hom(U,B)$ gives the desired natural transformation $G$.

Next, we make a map $G_{S} \colon \Hom(S,\P^{2}) \to \Hom(S,B)$ for each scheme $S$ of finite type over $\Bbbk$. 
Let $f \colon S \to \P^{2}$ be a morphism and $S=\bigcup_{i} U_{i}$ an affine covering. 
Let $h_{U_{i}}:=G_{U_{i}}(f|_{U_{i}}) \colon U_{i} \to B$ for each $i$. 
Let $U_{ij}:=U_{i} \cap U_{j}$. 
Then $h_{U_{i}}|_{U_{ij}} 
= \tau_{U_{i}}([f_{U_{i}},e_{U_{i}}])|_{U_{ij}} 
= \tau_{U_{ij}}([f_{U_{i}},e_{U_{i}}]|_{U_{ij}}) 
= \tau_{U_{ij}}([f_{U_{ij}},e_{U_{ij}}])
= h_{U_{j}}|_{U_{ij}}$. 
Hence we obtain a morphism $h \colon S \to B$ such that $h|_{U_{i}}=h_{i}$. 
Then we can define $G_{S}(f)$ as $h$. 

Finally, we check $G_{S}$ gives a functor $G \colon \Hom(-,\P^{2}) \to \Hom(-,B)$. 
For each morphism $a \colon S \to T$ in $(\Sch/\Bbbk)$, it suffices to show the following diagram commutes:
\[\xymatrix{
\Hom(T,\P^{2}) \ar[r]^{G_{T}} \ar[d]_{- \circ a} & \Hom(T,B) \ar[d]^{- \circ a} \\
\Hom(S,\P^{2}) \ar[r]_{G_{S}} & \Hom(S,B).
}\]
Let $f \colon T \to \P^{2}$ be a morphism. 
Take affine coverings $T=\bigcup_{i} T_{i}$ and $a^{-1}T_{i}=\bigcup_{j} S_{ij}$ for each $i$. 
Let $a_{ij}:=a|_{S_{ij}} \colon S_{ij} \to T_{i}$:
\[\xymatrix{
\Hom(T_{i},\P^{2}) \ar[r]^{G_{T_{i}}} \ar[d]_{- \circ a_{ij}} & \Hom(T_{i},B) \ar[d]^{- \circ a_{ij}} \\
\Hom(S_{ij},\P^{2}) \ar[r]_{G_{S_{ij}}} & \Hom(S_{ij},B).
}\]
Since the above diagram commutes, 
$(G_{T}(f) \circ a)|_{S_{ij}} = G_{T_{i}}(f|_{T_{i}}) \circ a_{ij} = G_{S_{ij}}(f|_{T_{i}} \circ a_{ij}) = G_{S}(f \circ a)|_{S_{ij}}$ for each $i,j$. 
Thus $G_{T}(f) \circ a = G_{S}(f \circ a)$. 
Thus $\P^{2}$ is the course moduli space of $\mcM^{\wF}_{0,1}$. 
We complete the proof. 
	
(3) Note that for every instanton bundle $\mcE$ with $c_{2}(\mcE)=2$, $\mcE(1)$ is $0$-regular by \cite[Lemma~3.1]{Kuznetsov12}. 
Hence, for a rank $2$ vector bundle $\mcE$ on $X$ with $c_{1}(\mcE)=0$ and $c_{2}(\mcE)=2$, $\mcE$ is weak Fano if and only if $\mcE$ is an instanton bundle. 
Therefore, the moduli space $M^{\ins}_{0,2}$ constructed in \cite[Section~5]{Sanna17} is the course moduli space $M^{\wF}_{0,2}$. 
Moreover, this moduli space is smooth, irreducible, and not fine by \cite[Theorem~5.8 and Proposition~5.12]{Sanna17}. 

We also recall that, for a rank $2$ vector bundle $\mcE$ on $X$ with $c_{1}(\mcE)=0$ and $c_{2}(\mcE)=3$, $\mcE$ is weak Fano if and only if $\mcE$ is an instanton bundle and $\mcE(1)$ is globally generated by Proposition~\ref{prop-FHI}.
In \cite[Section~6]{Sanna17}, Sanna constructed the fine moduli space $M^{\ins}_{0,3}$ of instanton bundles on $X$ of charge $3$, which is a smooth irreducible variety by \cite[Lemma~6.22 and Theorem~6.23]{Sanna17}. 
Thus we obtain the fine moduli space $M^{\wF}_{0,3}$ as an open subscheme of $M^{\ins}_{0,3}$. 
\end{proof}

\providecommand{\bysame}{\leavevmode\hbox to3em{\hrulefill}\thinspace}
\providecommand{\MR}{\relax\ifhmode\unskip\space\fi MR }
\providecommand{\MRhref}[2]{%
  \href{http://www.ams.org/mathscinet-getitem?mr=#1}{#2}
}
\providecommand{\href}[2]{#2}

\end{document}